\documentclass{mcom-l}

\usepackage{amssymb}

\usepackage{graphicx}

\usepackage{url}
\usepackage{algorithm,algorithmic}
\usepackage[dvipsnames]{xcolor}

\def\ba{\mathbf{a}}

\def\bd{\mathbf{d}}

\def\bg{\mathbf{g}}

\def\br{\mathbf{r}}
\def\bs{\mathbf{s}}

\def\bw{\mathbf{w}}
\def\bx{\mathbf{x}}
\def\by{\mathbf{y}}

\def\beps{\boldsymbol{\varepsilon}}

\def\bxi{\boldsymbol{\xi}}

\def\mbE{\mathbb{E}}

\def\mD{\mathcal{D}}
\def\mF{\mathcal{F}}

\def\mK{\mathcal{K}}
\def\mM{\mathcal{M}}
\def\mN{\mathcal{N}}

\def\mT{\mathcal{T}}

\DeclareMathOperator{\minimize}{minimize}

\renewcommand{\Re}{\mathbb{R}}

\renewcommand{\tilde}{\widetilde}
\renewcommand{\bar}{\overline}

\graphicspath{{./},{./Images/}}

\newtheorem{theorem}{Theorem}[section]
\newtheorem{lemma}[theorem]{Lemma}

\theoremstyle{definition}

\newtheorem{assumption}[theorem]{Assumption}

\theoremstyle{remark}
\newtheorem{remark}[theorem]{Remark}

\numberwithin{equation}{section}

\begin{document}

% \title[short text for running head]{full title}
\title[Line-search 2nd-Order Stochastic Nonconvex Optimization Methods]{LSOS: Line-search Second-Order Stochastic optimization methods for nonconvex finite sums}
\thanks{Research supported by the Executive Programme of Scientific and Technological
        Cooperation between the Italian Republic and the Republic of Serbia for years 2019-21
        (Italian grant no.~RS19MO05, Serbian grant no.~451-03-9/2021-14/200125). D.~di Serafino and M.~Viola were also supported
        by the Istituto Nazionale di Alta Matematica - Gruppo Nazionale per il Calcolo Scientifico (INdAM-GNCS) and by
        the 2019 V:ALERE Program of the University of Campania ``L. Vanvitelli''.}
%    Only \author and \address are required; other information is
%    optional.  Remove any unused author tags.

%    author one information
% \author[short version for running head]{name for top of paper}
\author[D. di~Serafino]{Daniela di~Serafino}
\address{Department of Mathematics and Applications ``Renato Caccioppoli'',
    University of Naples Fe\-de\-rico II, via Cintia, Monte S.~Angelo, 80126 Napoli, Italy}
\email{daniela.diserafino@unina.it}

%    author two information
\author[N. Kreji\'c]{Nata\v{s}a Kreji\'c}
\address{Department of Mathematics and Informatics, Faculty of Science,  University of Novi Sad,
    Trg Dositeja Obradovi\'ca 4, 21000 Novi Sad, Serbia}
\email{natasak@uns.ac.rs,natasa.krklec@dmi.uns.ac.rs}

%    author three information
\author[N. Krklec~Jerinki\'c]{Nata\v{s}a Krklec~Jerinki\'c}
%\address{Department of Mathematics and Informatics, Faculty of Science,  University of Novi Sad,
%    Trg Dositeja Obradovi\'ca 4, 21000 Novi Sad, Serbia}
%\email{natasa.krklec@dmi.uns.ac.rs}

%    author four information
\author[M. Viola]{Marco Viola}
\address{Department of Mathematics and Physics,
    University of Campania ``Luigi Vanvitelli'', viale A.~Lincoln 5, 81100 Caserta, Italy}
\email{marco.viola@unicampania.it}

%    \subjclass is required.


\date{}

%    Abstract is required.
\begin{abstract}
    We develop a line-search second-order algorithmic framework for minimizing finite sums. We do not make any convexity assumptions, but require the terms of the sum to be continuously differentiable and have Lipschitz-continuous gradients. The methods fitting into this framework combine line searches and suitably decaying step lengths. A key issue is a two-step sampling at each iteration, which allows us to control the error present in the line-search procedure. Stationarity of limit points is proved in the almost-sure sense, while almost-sure convergence of the sequence of approximations to the solution holds with the additional hypothesis that the functions are strongly convex. Numerical experiments, including comparisons with state-of-the art stochastic optimization methods, show the efficiency of our approach. 
\end{abstract}

\maketitle
\centerline{\footnotesize DATE: June 15, 2022}

\section{Introduction\label{sec:introduction}}

We are interested in solving large finite-sum minimization problems, where the objective function is, e.g., the sample mean of a finite family of possibly nonconvex smooth functions. This is the case of many statistical learning problems, including deep learning and more generally machine learning problems (see, e.g., \cite{Goodfellow:2016book, bottou:2018}), which have received much attention in the last years.
% convex problems: problems in which the logistic loss, the quadratic loss or other loss functions are used,
% often coupled with smooth regularization terms.
Specifically, we consider problems of the form
\begin{equation} \label{eqn:finite_sum}
    \underset{x \in \Re^n}{\minimize} \; \phi(\bx), \quad \phi(\bx) = \frac{1}{N} \sum_{i=1}^{N}\phi_i(\bx),
\end{equation}
where $\phi(\bx)$ is bounded from below and each $\phi_i(\bx)$ is twice continuously differentiable with Lipschitz-continuous gradient.
% Further on, we will add the assumption that the functions $ \phi_i(\bx) $ are  strongly convex, which will allow us to prove stronger convergence results.
Assuming that $ N $ is large, the computation of the objective function and its gradient and Hessian is expensive, and approximations of them are generally used. Subsampling is a natural way of computing these approximations, i.e., for a randomly chosen subsample $ \mM \subset \{1,\ldots,N\} $ with cardinality $ M $ we can approximate $ \phi(\bx) $ and its gradient and Hessian with the functions
\begin{equation} \label{subsample} 
    \begin{array}{rcl}
        \displaystyle f_\mM (\bx)      &=& \displaystyle \frac{1}{M}  \sum_{i \in \mM} \phi_i (\bx), \\
        \displaystyle \bg_\mM(\bx)  &=&  \displaystyle \frac{1}{M}  \sum_{i \in \mM} \nabla\phi_i(\bx), \\
        \displaystyle B_\mM (\bx)    &=&  \displaystyle \frac{1}{M} \sum_{i \in \mM} \nabla^2\phi_i(\bx),
    \end{array}
\end{equation}
respectively. However, other choices are possible and will be actually used in this work.

The general algorithmic scheme we consider is a second-order scheme with line search. In the deterministic setting, line-search second-order methods are successful in terms of global convergence and convergence rate. Most of these methods preserve their local rate of convergence as the step length tends to be $ 1 $ when the iterations come close to the solution. Hence, there is a good tradeoff between benefits and cost, in particular if one applies some second-order method with at least superlinear local convergence. In the stochastic framework the situation is not that simple. A key challenge appears to be the analysis of the gradient error at the new iterate $ \bx_k + t_k \bd_k $, because of the dependence between the step length $ t_k $ and the search direction $ \bd_k $. Since $ t_k $ is not fixed a priori, the line search introduces non-martingale dependencies, which do not allow us to easily estimate the error 
\begin{equation} \label{funerr} 
    | f_{\mN_k}(\bx_k + t_k \bd_k) - \phi(\bx_k + t_k \bd_k)| ,
\end{equation}
even if $ \mN_k $ is generated as an independent identically distributed (i.i.d.) sample at each iteration.  Thus, we propose a class of methods where this error is controlled by additional sampling. For other possibilities, see, e.g.,~\cite{iusem:2019} and ~\cite{krejic:2015oms}.

Stochastic optimization methods exploiting search directions based on second-order information have been widely investigated to get better theoretical and practical convergence properties than first-order stochastic methods, for either finite sums or general stochastic problems, especially when badly-scaled problems must be solved. Stochastic versions of Newton-type methods are discussed, e.g., in~\cite{bellavia:2019imajna,bollapragada:2019,byrd:2011,byrd:2012,curtis:2020,ruppert:1985,spall:1994,spall:1995,spall:1997},
% and a second-order subsampled trust region method is presented in~\cite{bellavia:2020coap}.
%and a variant of the adaptive cubic regularization scheme using a dynamic rule for building inexact Hessian information is proposed in~\cite{bellavia:2020imajna}.
and variants of the adaptive cubic regularization schemes are proposed in~\cite{bellavia:2020imajna,park:2020}.
In particular, stochastic quasi-Newton methods are analyzed, either in the strongly convex or in the nonconvex setting, in~\cite{berahas:2020,byrd:2011,byrd:2016,chen:2020,gower:2016,berahas:2021,mokhtari:2014,mokhtari:2015,moritz:2016,wang:2017siopt,wang:2019}.

For strongly convex functions of type~\eqref{eqn:finite_sum}, in~\cite{moritz:2016} Moritz et al.~propose a stochastic L-BFGS algorithm based on the same inverse Hessian approximation as in~\cite{byrd:2016}, but use SVRG~\cite{johnson:2013svrg} instead of the standard stochastic gradient approximation (see, e.g., \cite{bottou:2018}). This L-BFGS algorithm, which applies a constant step length, has Q-linear rate of convergence of the expected value of the error in the objective function. A modification to this L-BFGS scheme is proposed by Gower et al.~in~\cite{gower:2016}, where a stochastic block BFGS update is used, in which the vector pairs for updating the inverse Hessian are replaced by matrix pairs gathering directions and matrix-vector products between subsampled Hessians and those directions. The resulting algorithm uses a constant step length and has Q-linear convergence rate of the expected value of the objective function error, as in the previous case, but appears more efficient by numerical experiments.
More recently, the Incremental Quasi-Newton (IQN) method~\cite{mokhtari:2018} has been developed, which has been proved to have local superlinear convergence. IQN differs from other stochastic quasi-Newton methods as it uses aggregated information on variables, gradients, and quasi-Newton Hessian approximations to reduce the noise of gradient and Hessian approximations, applies a cyclic scheme to update the functions, and  approximates each individual function by a Taylor's expansion where the linear and quadratic terms are evaluated with respect to the same iterate. Nevertheless, despite we do not use such a combination of techniques, we will show that the stochastic L-BFGS method presented in this work, which exploits the Jacobian sketching described in~\cite{gower:2020}, in practice is often faster than IQN on widely used test problems.

When dealing with nonconvex finite sums, the (L-)BFGS updates may lead to indefinite Hessian approximations, which in turn may prevent convergence. To cope with this problem, specifically suited versions of stochastic L-BFGS algorithms were developed -- see, e.g., \cite{berahas:2020,wang:2017siopt}. In particular, in~\cite{wang:2017siopt} the authors propose to use damping to guarantee positive definiteness of the Hessian approximations, and combine the modified L-BFGS update with the SVRG gradient approximation. The resulting algorithm, named SdLBFGS-VR, equipped with a constant step length, is guaranteed to bring the expected value of the gradient norm (computed among all the iterations) below a predefined threshold in a finite number of steps. This method is used for comparison purposes here.

\subsection*{Contribution and outline of this work\label{sec:contribution}}

We propose a Line-search Second-Order Stochastic (LSOS) algorithmic framework for problem~\eqref{eqn:finite_sum}, where Newton and quasi-Newton directions in a rather broad meaning are used. Inexactness is allowed in the sense that the Newton-type direction can be obtained as inexact solution of the corresponding system of linear equations. The objective function is approximated by subsampling, while other approximations can be used for the gradient and the Hessian. We consider a two-step sampling at each iteration, which allows us to control the error~\eqref{funerr} introduced by the line-search procedure. {The second sampling can be of arbitrary small size, even the sample size 1 is sufficient,  and hence it does not increase the computational costs significantly. If the proposed line search is unsuccessful in some prefixed (possibly very large) number of iterations, the algorithm switches to predefined step sizes, i.e., to the SA method.} 

We prove that limit points are stationary in the almost sure sense for bounded functions with Lipschitz-continuous gradients, while the sequence of approximations to the solution converges almost surely in the case of strongly-convex functions. Although we cannot prove that our class of methods has superlinear convergence rate, we show by numerical experiments that these methods are competitive or faster than state-of-the art second-order stochastic optimization methods. {Furthermore, our preliminary experiments show that the line search steps are accepted almost always and the switching point, from the line search to the SA method, is never reached.} 
An additional advantage of the proposed method is that it can be extended to a more general class of problems where the sample is infinite (for example, online training) by performing simple modifications in the LSOS algorithm (see Remark~\ref{rem:infinite_sample}).

The paper is organized as follows. In Section~\ref{sec:lsos}, we introduce the LSOS framework for finite sum problems. In Section~\ref{sec:convergence}, we prove almost sure convergence to a stationary point of any algorithm fitting into the general LSOS framework. In Section~\ref{sec:lbfgs-saga}, we present a specialization of the LSOS framework, named LSOS-BFGS, which exploits a mini-batch variant of the SAGA algorithm~\cite{defazio:2014saga}, used in~\cite{gower:2020}, and approximates the inverse Hessian by means of a stochastic version of the limited-memory BFGS (L-BFGS) proposed in~\cite{byrd:2016}. Moreover, for nonconvex objective functions, we propose a modified version of the L-BFGS update in~\cite{byrd:2016}, which is inspired by the damping strategy used in~\cite{wang:2017siopt}. In Section~\ref{sec:experiments}, we compare a MATLAB implementation of the LSOS-BFGS algorithm with some state-of-the-art methods in the solution of both nonconvex and convex problems of the form \eqref{eqn:finite_sum}. Finally, we draw some conclusions in Section~\ref{sec:conclusions}.

\subsection*{Notation}

$\mbE(x)$ denotes the expectation of a random variable $x$ %,  $\mathrm{var}(x)$ the variance of $x$,
and $\mbE(x | \mF)$ the conditional expectation of $x$ for a given $\sigma$-algebra $\mF$, where the $\sigma$-algebra is determined by random variables precisely defined in the sequel.  $ \|\cdot\| $ indicates either the Euclidean vector norm or the corresponding induced matrix norm. %, while $| \cdot |$ is the cardinality of a set.
$\Re_+$ and $\Re_{++}$ denote the sets of real non-negative and positive numbers, respectively, and $ \mN = \{1,\ldots,N\} $. Vectors are written in boldface and subscripts indicate the elements of a sequence, e.g., $\{ \bx_k \}$. Furthermore, $ \bg(\bx) $ and $ B(\bx) $ denote approximations of the gradient and the Hessian of $ \phi $ at $ \bx $; we also use $ \bg_k $ and $ B_k $ when $ \bx = \bx_k$. Finally, ``a.s.'' abbreviates ``almost sure/surely''.

\section{Algorithmic framework\label{sec:lsos}}

We consider second-order methods, where the search direction is found by solving a Newton-type system of linear equations. The dimension of the system may be very large and thus a natural choice is to solve it inexactly, e.g., using some iterative procedure. We consider search directions $ \bd_k $ such that the following inexact Newton condition is satisfied:
\begin{equation} \label{eq:INerror}
    \|B_k \bd_k + \bg_k\| \leq \delta_k \|\bg_k\| ,
\end{equation}
where $ \bg_k $ may be, e.g., $ \bg_{\mN_k}(x_k) $ for a certain sample $ \mN_k $. We assume that the approximate gradient $ \bg_k$  and the approximate Hessian $ B_k $ are conditionally independent. This assumption will be discussed later and we will show that quasi-Newton methods satisfy it in the case of finite-sum problems under standard conditions. 

Our class of methods fits into the framework of Algorithm~\ref{alg:LSOS}. It combines the line-search with the Stochastic Approximation (SA) approach~\cite{robbins:1951}, where the step-length sequence $ \{\alpha_k\} $ is non-summable but square-summable. The main drawback of the latter step lengths is that they quickly become very small and thus the convergence may be extremely slow. On the other hand, the line search introduces non-martingale errors, which are difficult to estimate and bound. The key idea in our algorithm is to combine the two approaches to get a.s.~convergence, {but keeping the advantage of hopefully large step sizes coming from the line search at least at initial iterations}. Additional sampling is performed within the algorithm to control the non-martingale errors. 

\begin{algorithm}[h!]
    \caption{Line-search Second-Order Stochastic (LSOS) method\label{alg:LSOS}}
    {\small
        \begin{algorithmic}[1]
            \STATE given $ \bx^0 \in \Re^n $, $ B_0 \in \mathbb{R}^{n \times n}, $ $ \eta, \beta \in (0,1) $, % $N_0 \ll N$,
            $ \{\alpha_k\}, \{\zeta_k\} \subset \Re_{++} $, $ \{\delta_k\} \subset \Re_{+} $, $ c_{\min}, C_{\max} \in \Re_{+} $, $ K_{\max}  \in \mathbb{N}$
            \STATE $ K_f = 0 $, $ ind = 0 $, $ k = 0 $
            \WHILE {stop criterion not satisfied}
            \STATE choose $ \mN_k \subset \mN $ randomly and uniformly and compute $ \bg_k $ \label{alg:sampling_1} 
            \STATE find a search direction $\bd_k $ such that \label{alg:inexact_direction} %\\[-15pt]
            \begin{equation} \label{eqn:alg_finite_sum_inexact_direction}
                \| B_k \bd_k + \bg_k \| \leq \delta_k \| \bg_k \|
            \end{equation}
            %\par \vspace*{-5pt}
            \IF {$ ind = 0 $}
            \STATE find the smallest integer $j \ge 0$ such that the step length $t_k=\beta^j$ satisfies \label{alg:line_search} %\\[-15pt]
            \begin{equation} \label{eqn:alg_finite_sum_line_search}
                f_{\mN_k}(\bx_k + t_k \bd_k) \leq f_{\mN_k}(\bx_k) + \eta\,t_k\, \bg_k^\top \bd_k + \zeta_k
            \end{equation}
            %\par \vspace*{-5pt}
            \STATE $ \bar \bx_k = \bx_k + t_k \bd_k $
            \STATE choose $ \mD_k \subset \mN $ randomly and uniformly \label{alg:sampling_2}
            \IF {$ \displaystyle f_{\mD_k}(\bar \bx_k) \le  f_{\mD_k}(\bx_{k}) - c_{\min} \|\bg_{\mD_k}(\bx_{k})\|^2 + C_{\max} \, \zeta_k $ \label{alg:decr_cond_2}}
            \STATE $ \bx_{k+1} = \bar \bx_k $ \label{alg:new_iterate_ls}
            \ELSE
            \STATE $ \bx_{k+1} = \bx_k $
            \STATE $ K_f = K_f + 1$ 
            \IF {$ K_f > K_{\max} $}
            \STATE $ ind = 1$
            \ENDIF
            \ENDIF
            \ELSE
            \STATE {$ t_k = \alpha_k $}
            \STATE {$ \bx_{k+1} = \bx_k + t_k \bd_k $} \label{alg:new_iterate_sa}
            \ENDIF
            \STATE compute $ B_{k+1} $.  \label{alg:new_hessian}
            \STATE $ k = k + 1 $
            \ENDWHILE
        \end{algorithmic}
    }
\end{algorithm}

We comment Algorithm~\ref{alg:LSOS} in detail. At the beginning of each iteration $ k+1 $, the quasi-Newton matrix $ B_k $ and the iterate $ \bx_k$ are available. First, we generate the new sample {$ \mN_k $ and compute $ \bg_k. $ Notice that we do not impose any assumption on $ \mN_k $ except the unbiasedness and the finite variance of the gradient approximation -- see Assumption \ref{ass:gradient} ahead.  The search direction $ \bd_k $ is computed such that ~\eqref{eqn:alg_finite_sum_inexact_direction} holds. We then proceed to the step-size computation. The variable $ ind $ is governing the switch between the line-search step length and the SA step length; as long as $ ind = 0 $ we perform a line search to get the step size $ t_k $ such that the nonmonotone Armijo condition (\ref{eqn:alg_finite_sum_line_search}) is satisfied. That way we get a new candidate point $\bar \bx_k. $ Having the candidate point we perform the additional sampling, generating a sample $ \mD_k$ uniformly and randomly. Notice that there are no conditions on the size of $ \mD_k,$ i.e., it can be a sample of size 1; then we check if the candidate point $\bar \bx_k $ satisfies the decrease condition for the approximate function $ f_{\mD_k} $. In this case, we accept the update $ \bx_{k+1} = \bar \bx_k, $ the line search is considered to be successful and we proceed to generate the  Hessian approximation $ B_{k+1} $ that will be used in the next iteration. 

In the case of insufficient decrease for the approximate function $ f_{\mD_k} $, the line search step is not successful, we discard the candidate point $\bar \bx_k  $ and declare $ x_{k+1} = x_k. $ The counter $ K_f $ is increased. As long as $ K_f < K_{\max} $ we keep with the line search procedure. However if $K_{\max}$ is reached we change the indicator variable to $ ind=1$ and the algorithm switches to the predefined step sizes $ \alpha_k,$ i.e.~to the SA method.

The additional check at line 10 is one of the key novelties of the LSOS algorithm and the reasoning behind this check is the following.  Under the assumptions stated in the next section, one can prove that the step size $t_k$ in the line search is bounded from below by $t_{\min}:=\beta (1-\eta) \mu^2/(L^2 (1+\delta_{max})^2),$ provided that $\delta_k \leq \mu / (2L)$ and $ \bg_k=\bg_{\mN_k}(\bx_k).$ In that case, there holds 
$$f_{\mN_k}(\bar{\bx}_k) \leq  f_{\mN_k}(\bx_k) - c \|\bg_k\|^2,  $$
where $c:=t_{\min} \eta / (2 L).$ In other words, this line search provides sufficient decrease with respect to $f_{\mN_k}$. Thus, we check if this new point would be of similar quality to the independently chosen subsample $\mD_k$ and the corresponding function $f_{\mD_k}$ with a deterioration controlled by $\zeta_k$.  Implicitly, we can consider this as a check of similarity of the functions $f_i$. If these functions are similar, then performing a line search on a subsampled function is probably beneficial since it is a good, but cheaper, approximation of the objective function $\phi$. Otherwise, the dissimilarity of the functions is too big and the line-search step is not successful. If the unsuccessful steps occur sufficiently many times ($K_{max}$ times), the algorithm is switched to the SA phase, which may be slower, but more stable for that kind of problems. This reasoning helps us to take the advantages of both the SA and LS variants, having a theoretically sound algorithm with good practical behavior.

We note that both decrease rules (lines \ref{alg:line_search} and \ref{alg:decr_cond_2} of the algorithm) are nonmonotone. The reason is that the inexact direction $\bd_k$ does not need to be a decreasing direction for $f_{\mN_k}$, but it might be a good direction for the original objective function. The term $ \zeta_k $ makes the step always well defined as for $ t_k $ small enough one can always satisfy the condition \eqref{eqn:alg_finite_sum_line_search} and hence we have the finite termination  of the backtracking loop, see Lemma~\ref{le:finite} below. Thus the algorithm is well defined. In the case $ \bd_k $ is a decreasing direction, the nonmonotone rule allows us to take larger steps. As regards the decrease condition at line \ref{alg:decr_cond_2}, we notice  that the lack of decrease in $ f_{\mD_k} $  is not necessarily an increase in the function value and given that $ C_{\max}, \, \zeta_k > 0 $, we can regard the condition at line~\ref{alg:decr_cond_2} of Algorithm~\ref{alg:LSOS} as a nonmonotone line-search condition. However, as specified in the next section, we have $ \zeta_k \to 0 $ and thus this condition becomes stricter and stricter.

\begin{lemma} \label{le:finite}
Let $ f_i, \, i=1,\ldots,N$, be continuous. Given $ \bd_k, \bg_k \in {\mathbb R}^n, $  and $ \zeta_k > 0 $ the backtracking defined at line 7 of the LSOS algorithm has finite termination.  
\end{lemma}
    
\begin{proof}
The function $f_{\mN_k} $ is continuous, $ \zeta_k > 0 $ and  for $ j $ large enough we can always take  $ t_k = \beta^j $  such that  $  t_k \bd_k^T \bg_k $ and $ t_k \bd_k $ are sufficiently small and therefore  $ f_{\mN_k}(\bx_k + t_k \bd_k) -  f_{\mN_k}(\bx_k) - \eta t_k \bd_k^T \bg_k \leq \zeta_k. $ 
\end{proof}

The integer $ K _{\max} $ is arbitrarily large, but fixed, and controls the maximum number of iterations in which we might have an increase in the additional sampled function $f_{\mD_k},$ i.e. the maximal number of unsuccessful line-search iterations. At the end of each iteration, we compute the new approximate Hessian $ B_{k+1}$.

In the next section we show that this algorithm converges almost surely. Summarizing the properties of the algorithm, we can see that there are two possible scenarios. The first one is that the number of unsuccessful line-search steps is smaller than $ K_{\max} $ when the stopping criteria is reached and hence the method has generated an iterative sequence with the line-search procedure. The second possibility is that $ K_f = K_{\max} $ before the stopping criteria is reached and we have that, starting from a given point, the iterate sequence has been generated by SA. Clearly, the role of $ K_{\max} $ is essential as its value determines the properties of the sequence generated by LSOS. We observe that it is very likely that a suitable value of $ K_{\max} $ can be problem dependent. Nevertheless, by choosing a relatively large $ K_{\max} $, one can enforce line-search steps which yield larger step sizes if successful and hence faster convergence. Of course it might happen that many line-search steps are unsuccessful, thus resulting in a waste of time for the algorithm. Anyway, all our experiments, see Section~\ref{sec:experiments}, indicate that the line search iterations are successful in a vast majority of cases, the number of discarded candidate points is in the interval 1-6\%, and $ K_{\max} $ is never reached.}

{An additional question that may arise here regards the second scenario. Assume that $ K_{\max} $ is reached and thus we switch to the SA step sizes. A number of successful methods with approximate gradients and variance reduction are defined for strongly convex problems (see, e.g., \cite{johnson:2013svrg,defazio:2014saga,gower:2020}). Some of these methods work with fixed step sizes and Algorithm~\ref{alg:LSOS} implies decreasing step sizes. Clearly it would be better to use a fixed step size in this scenario, assuming that the search direction $ \bd_k $ satisfies the variance-reduction properties. One could reformulate Algorithm 1 such that it covers this possibility. However such a reformulation would imply specifying a number of additional assumptions on the gradient approximation as well as on the construction of the Hessian approximation $ B_k $. Our intention here was to propose a rather general scheme, so we did not consider this case separately. But we tested the method against the fixed step-size methods with variance reduction in Section~\ref{sec:experiments} and demonstrated the advantages of the proposed algorithm. Given that we never came close to $ K_{\max} $, the second scenario did not occur in our tests. }

{It is worth mentioning that in \cite{paquette:2020} and \cite{berahas:2021siopt} the authors also consider line-search procedures in the stochastic framework. In both cases the key assumptions are that the sequences of random estimates of the function and the gradient are probabilistically accurate in the submartingale sense. In the case of subsampled functions and gradients the condition reduces to taking the size of $\mN_k$ large enough to be able to satisfy the required accuracy for both the function and the gradient. Moreover, we note that in both cases, while an Armijo-like condition is checked (in \cite{berahas:2021siopt} the authors use a nonmonotone condition relying on an a-priori knowledge of the function evalution accuracy), no formal backtracking is performed. In fact, if the line-search condition is not satisfied by the initial step length, then the authors propose to immediately reject the iterate, shrink the step length and recompute the gradient and function approximations from scratch. These ingredients allow the authors to develop stochastic Armijo-like line-search methods that need neither the additional sample $ \mD_k $ nor switching to SA in any scenario. We note that the sample size fulfilling the aforementioned probabilistic assumptions is rather large while we do not impose any condition on $ \mN_k $ besides unbiasedeness and finite variance of the gradient approximation. Hence, it is quite difficult to compare the approach we propose here with the ones in~\cite{paquette:2020,berahas:2021siopt}. Furthermore, the complexity results given in the mentioned works rely on fixed probabilities of the estimated function and gradient, while we do not present formal complexity results here. To offer some insight into the complexity of the proposed method, in particular with respect to the possible increase of the oracle complexity due to the independent sampling of $ \mN_k $ and $ \mD_k $, in Section~\ref{sec:experiments} we provide an empirical analysis in terms of oracle complexity for LSOS, comparing it with state-of-the-art methods in terms of number of data accesses.}

\section{Convergence theory\label{sec:convergence}}

We state more formally our assumptions on the minimization problem and on some quantities used in Algorithm~LSOS.

\begin{assumption} \label{ass:phi_properties}
    The function $ \phi $ is bounded from below by $ \phi^*$ and the functions $ \phi_i $ have Lipschitz-continuous gradients with Lipschitz constant $L$.
\end{assumption}
Although we do not suppose $ \phi $ is strongly convex, we make the following assumption on the approximate Hessians computed by the algorithm. Without loss of generality we take the same $L$ as in the previous assumption. 
\begin{assumption} \label{ass:B(xk)}
    There exist $ \mu, L > 0 $ such that
    \begin{equation*} % \label{ass:hessian_bounds}
        \mu I \preceq B_k  \preceq L I
    \end{equation*}
    for all $k$. 
\end{assumption}

We also specify some properties of the gradient approximation. To this aim, we denote by $ \beps_g(\bx) $ the error in the approximation of $ \nabla \phi(\bx) $:
\begin{equation} \label{eqn:gradient1} 
    \nabla \phi(\bx) = \bg(\bx) + \beps_g(\bx) ,
\end{equation}
and by $ \mF_k $ the $ \sigma$-algebra identified by $ \{\bx_0,\bx_1,\ldots,\bx_k\} $.
\begin{assumption} \label{ass:gradient}
    There exists a constant $ \bar c_1 > 0 $ such that
    \begin{equation} \label{eqn:gradient2}
        \mbE(\bg_k(\bx_k) | \mF_k) = \nabla \phi(\bx_k) \; \mbox{ and } \; \mbE(\|\varepsilon_{g_k}(\bx_k)\|^2|\mF_k) \leq \bar c_1 . % < \infty. 
    \end{equation}
    for all $k$.
\end{assumption}
\noindent
In other words, we assume that the expected gradient noise is zero and the variance of the gradient errors
% , {\color{magenta} Can we drop  the variance formulas and just to say that the variance is bounded?  }
% $$
%    \mathrm{var}(\|\beps_g(\bx)\|\, | \mF_k) = \mbE(\|\beps_g(\bx)\|^2\, | \mF_k) - \mbE^2(\|\beps_g(\bx)\|\, | \mF_k),
% $$
is bounded. % var > 0

Finally, we make some (standard) assumptions on the sequences $ \{\alpha_k\}, \{\zeta_k\} \subset \Re_{++} $ and $ \{\delta_k\} \subset \Re_{+} $.
\begin{assumption} \label{ass:sequences}
    \begin{eqnarray*}
        % \label{eqn:zeta_k_summable}
        \sum_{k=1}^{\infty} \zeta_k < \infty , & &\\
        % \label{eqn:alpha}
        \sum_{k=0}^{\infty} \alpha_k = \infty, & \quad & \sum_{k=1}^{\infty} \alpha_k^2 < \infty , \\
        % \label{eqn:deltaalpha}
        \delta_k \to 0, \qquad & \quad &  \sum_{k=1}^{\infty} \delta_k \alpha_k < \infty.
    \end{eqnarray*} 
\end{assumption}

%A few comments are in order. 
Let us recall a few properties of the algorithm here. In the first iterations of the algorithm, non-descent directions are likely to occur; however, by requiring $\zeta_k>0$ we ensure that the line search remains well defined, {by Lemma~\ref{le:finite}}.  Furthermore, by~Assumption~\ref{ass:sequences} it is $\zeta_k \to 0$, which implies that Algorithm~\ref{alg:LSOS} eventually determines a descent direction for the current approximation of the objective function. At the $ k $-th iteration, we may reject the update $ \bar \bx_{k+1} = \bx_k + t_k \bd_k $ obtained with the line search, if the condition at line~\ref{alg:decr_cond_2} of the algorithm is not satisfied. However, only a finite number of rejections, specified by $K_{\max}$, is allowed. In any case the new approximate Hessian $ B_{k+1} $ is computed at the end. The $ (k+1) $-st iteration starts with sampling at line~\ref{alg:sampling_1}, and the sample $ \mN_{k+1} $ is used to compute the approximations of the gradient and the function. Thus $ B_{k+1} $ and $ \bg_{k+1} $ are conditionally independent in the sense specified by the following lemma.

\begin{lemma} \label{le:independence}
    The approximate gradient  $ \bg_k$ and the approximate Hessian $ B_k $ generated by Algorithm~\ref{alg:LSOS} satisfy
    $$ \mbE(B_k \bg_k | \mF_k) = B_k \mbE(\bg_k | \mF_k). $$   
\end{lemma}

Next, we give some properties of the directions $\bd_k$ generated by Algorithm~\ref{alg:LSOS}. Throughout this section we use $ \displaystyle \delta_{\max} = \max_k \delta_k$.

\begin{lemma} \label{dklemma}
    Suppose that
    % Assumptions~\ref{ass:phi_properties} and
    Assumption~\ref{ass:B(xk)} holds and $\delta_k \leq \mu/(2L)$. Then
    \begin{equation} \label{eqn:ls_conv_descent_bound}
        \bg_k^\top \bd_k \leq -\frac{1}{2L} \|\bg_k\|^2
    \end{equation}
    and 
    \begin{equation} \label{dk2}
        \|\bd_k\|^2 \leq \frac{(\delta_{max}+1)^2}{\mu^2}\|\bg_k\|^2 .
    \end{equation}
\end{lemma}  

\begin{proof}
Let $ \br_k = B_k \bd_k + \bg_k $. We have
\begin{equation} \label{dkrk}
    \bd_k = B_k^{-1} \br_k - B_k^{-1} \bg_k. 
\end{equation}
Assumption~\ref{ass:B(xk)} together with~\eqref{eqn:alg_finite_sum_inexact_direction} implies
\begin{eqnarray*}
    \bg_k^\top \bd_k & = & \bg_k^\top  B_k^{-1} \br_k - \bg_k^\top  B_k^{-1} \bg_k \\
    & \leq  & \|\bg_k\| \|B_k^{-1}\| \| \br_k\| - \frac{1}{L} \|\bg_k\|^2 \\
    & \leq & \frac{1}{\mu} \delta_k \|\bg_k\|^2  - \frac{1}{L} \|\bg_k\|^2 \\
    & = & \left(\frac{\delta_k}{\mu}  - \frac{1}{L} \right)\|\bg_k\|^2.
\end{eqnarray*}
Since  $\delta_k \leq \frac{\mu}{2 L}$, we conclude that~\eqref{eqn:ls_conv_descent_bound} holds. Furthermore, for each $k$ we have 
\begin{equation} \label{dkbound}
    \|\bd_k\| =\|B_k^{-1} (\br_k - \bg_k)\|\leq \frac{1}{\mu}(\|\br_k\|+\|\bg_k\|)\leq \frac{\delta_k+1}{\mu}\|\bg_k\| ,
\end{equation}
thus, squaring and using $\delta_k \leq \delta_{\max}$ we obtain~\eqref{dk2}.
\end{proof}

The following theorem on the convergence of perturbed nonnegative supermartingales~\cite{robbins:1971} is used for proving the convergence when the step-length choice corresponding to $ ind = 1 $ is activated.
\begin{theorem} \label{thm:rand_variable_conv}
Let $ U_k, \beta_k, \xi_k, \rho_k \geq 0 $ be $ \mF_k$-measurable random variables such that
$$ \mbE(U_{k+1} | \mF_k) \leq  (1+\beta_k) U_k + \xi_k - \rho_k, \quad k=1,2,\ldots \, . $$
If $ \displaystyle \sum_k \beta_k < \infty $ and $ \displaystyle \sum_k \xi_k < \infty $, then $ U_k \to U < \infty $ a.s. and $ \displaystyle \sum_k \rho_k < \infty $ a.s..
\end{theorem}

\smallskip
{
Now we are ready to prove convergence results for Algorithm~LSOS. Given that two scenarios are possible, i.e.~$ K_{\max} $ is reached or not, we consider the following two theorems and then state the overall convergence result by combining them. 

\begin{theorem} \label{th:main1}
Let Assumptions~\ref{ass:phi_properties}-\ref{ass:sequences} hold, $ \{\bx_k\} $ be a sequence generated by Algorithm~\ref{alg:LSOS} and  assume that $ K_f = K_{\max} +1 $ is reached.  
Then
\begin{equation} \label{pr30}
    \liminf_{k \to \infty}\|\nabla\phi(\bx_k)\|=0 \quad \mbox{a.s.} \, . 
\end{equation}
\end{theorem}
}
\begin{proof}
%We distinguish two cases: i) $K_f\leq K_{max}$, i.e., the SA step length is never used; ii) $K_f> K_{max}$, i.e., the SA step length is eventually used. 
Given that $ K_f> K_{max} $, the SA step length is eventually chosen. Then there exists $\bar{k}$ such that the SA step length is chosen for all $k\geq \bar{k}$. Assumption~\ref{ass:gradient} implies 
\begin{eqnarray} \label{pr21} 
    \mbE(\|\bg_k\|^2 | \mF_k) & \leq & 2 \left( \mbE(\|\nabla \phi(\bx_k)\|^2 | \mF_k) + \mbE(\|\varepsilon_{g_k}(\bx_k)\|^2 | \mF_k) \right) \\ \nonumber
    & \leq & 2 \left( \|\nabla \phi(\bx_k)\|^2+\bar c_1 \right).
\end{eqnarray} 
Moreover, we have 
\begin{eqnarray*} % \label{pr22} 
    \mbE(\|\bg_k\| | \mF_k)&\leq& \sqrt{\mbE(\|\bg_k\|^2 | \mF_k)} 
    \leq \sqrt{2 \left( \|\nabla \phi(\bx_k)\|^2+\bar c_1 \right)}\\ \nonumber
    &\leq & \sqrt{2} \left( \|\nabla \phi(\bx_k)\|+\sqrt{\bar c_1} \right).
\end{eqnarray*} 
Recall that for all $k \geq \bar{k}$ we have $ \bx_{k+1} = \bx_{k}+\alpha_k \bd_k $, where $ \alpha_k $ is pre-determined and $\bd_k$ satisfies~\eqref{eqn:alg_finite_sum_inexact_direction} in algorithm~\ref{alg:LSOS}. From now on, we assume $k \geq \bar{k}$. Notice that Assumptions~\ref{ass:B(xk)} and~\ref{ass:gradient} together with Lemma~\ref{le:independence} applied to $B_k^{-1}$ imply 
\begin{eqnarray} % \label{pr23}
    \nabla\phi(\bx_k)^\top \mbE(  B_k^{-1} \bg_k| \mF_k)&=&\nabla\phi(\bx_k)^\top  B_k^{-1} \mbE( \bg_k| \mF_k)\\ \nonumber
    &=& \nabla\phi(\bx_k)^\top  B_k^{-1} \nabla\phi(\bx_k)\geq \frac{1}{L} \|\nabla\phi(\bx_k)\|^2.
\end{eqnarray}
This and~\eqref{dkrk} lead to
\begin{equation} \label{pr24} 
\begin{split}
    \mbE(\alpha_k \nabla\phi(\bx_k)^\top \bd_k  | \mF_k) &= \alpha_k \nabla\phi(\bx_k)^\top \mbE( \bd_k  | \mF_k)\\
    &= \alpha_k \nabla\phi(\bx_k)^\top \mbE(  B_k^{-1} \br_k - B_k^{-1} \bg_k| \mF_k)\\
    &\leq \alpha_k \left(\|\nabla\phi(\bx_k)\| \mbE(  \|B_k^{-1}\| \| \br_k \|  | \mF_k)- \frac{1}{L} \|\nabla\phi(\bx_k)\|^2\right)\\
    &\leq  \alpha_k \left(\|\nabla\phi(\bx_k)\| \mbE(   \frac{\delta_k}{\mu} \|\bg_k\|  | \mF_k)- \frac{1}{L} \|\nabla\phi(\bx_k)\|^2\right)\\
    &\leq  \alpha_k \|\nabla\phi(\bx_k)\|^2\left(\frac{\sqrt{2} \delta_k}{\mu}-\frac{1}{L}\right)
    + \alpha_k \|\nabla\phi(\bx_k)\| \frac{\sqrt{2\,\bar c_1} \delta_k}{\mu}.
\end{split}
\end{equation} 
Using Assumption~\ref{ass:phi_properties} and the descent lemma~\cite[Proposition~A24]{bertsekas:1999book}, we get
$$ \phi(\bx_{k+1})
\leq \phi(\bx_k)  + \alpha_k \nabla\phi(\bx_k)^\top \bd_k + \frac{1}{2} \alpha_k^2 L\, \|\bd_k\|^2.  $$
Applying the conditional expectation and using \eqref{pr24} and \eqref{dk2} we obtain 
\begin{eqnarray} \label{pr25} 
    \mbE(\phi(\bx_{k+1})  | \mF_k)
    &\leq&
    \mbE(\phi(\bx_k)| \mF_k)  + \alpha_k \|\nabla\phi(\bx_k)\|^2 \left( \frac{\sqrt{2} \delta_k}{\mu}-\frac{1}{L} \right) \\ \nonumber
    && +\alpha_k \|\nabla\phi(\bx_k)\| \frac{\sqrt{2\,\bar c_1} \delta_k}{\mu}\\\nonumber
    && + \frac{1}{2} \alpha_k^2 L \mbE\left(\frac{(\delta_{max}+1)^2}{\mu^2}\|\bg_k\|^2| \mF_k\right)\\\nonumber
    &\leq&
    \phi(\bx_k)  + \alpha_k \|\nabla\phi(\bx_k)\|^2\left(\frac{\sqrt{2} \delta_k}{\mu}-\frac{1}{L}\right)\\\nonumber
    && + \alpha_k \|\nabla\phi(\bx_k)\| \frac{\sqrt{2\,\bar c_1} \delta_k}{\mu}\\\nonumber
    && +\frac{1}{2} \alpha_k^2 L \frac{(\delta_{max}+1)^2}{\mu^2}2(\|\nabla \phi(\bx_k)\|^2+\bar c_1). 
\end{eqnarray}
Rearranging this inequality we obtain
\begin{equation} \label{pr26}
    \mbE(\phi(\bx_{k+1})  | \mF_k) \leq \phi(\bx_k)
    +\|\nabla\phi(\bx_k)\|^2 \alpha_k (c_k-\frac{1}{L})
    +\alpha_k^2 \bar c_2+\|\nabla\phi(\bx_k)\| \alpha_k \delta_k \bar c_3,
\end{equation}
where $$c_k=\frac{\sqrt{2} \delta_k}{\mu}+\frac{\alpha_k L (\delta_{max}+1)^2}{\mu^2}, \; \bar c_2=\frac{\bar c_1 L (\delta_{max}+1)^2}{\mu^2}, \; \bar c_3=\frac{\sqrt{2\,\bar c_1}}{\mu}.$$
\noindent
By % $\delta_k\rightarrow 0$ and 
Assumption~\ref{ass:sequences}, we have that $c_k \to 0$ and thus there exists $\tilde{k} \geq \bar{k}$ such that for all $k \geq \tilde{k}$ we have $c_k-1/L \leq -1/(2L)$. Henceforth, we assume $k \geq \tilde{k}$.  From~\eqref{pr26} we obtain 
\begin{equation} \label{pr27}
    \mbE(\phi(\bx_{k+1})  | \mF_k) \leq \phi(\bx_k)
    -\frac{1}{2L} \|\nabla\phi(\bx_k)\|^2 \alpha_k 
    +\alpha_k^2 \bar c_2+\|\nabla\phi(\bx_k)\| \alpha_k \delta_k \bar c_3,
\end{equation}
Let us consider two cases: i) $\|\nabla\phi(\bx_k)\|\leq 1$, and ii) $\|\nabla\phi(\bx_k)\|>1$.
If  $\|\nabla\phi(\bx_k)\|\leq 1$, then 
$$\|\nabla\phi(\bx_k)\| \alpha_k \delta_k \bar c_3\leq \alpha_k \delta_k \bar c_3.$$
If $\|\nabla\phi(\bx_k)\|>1$, then $\|\nabla\phi(\bx_k)\|\leq \|\nabla\phi(\bx_k)\|^2$ and 
$$\|\nabla\phi(\bx_k)\| \alpha_k \delta_k \bar c_3\leq \|\nabla\phi(\bx_k)\|^2\alpha_k \delta_k \bar c_3.$$
Thus, we can conclude that the following inequality holds in general: 
$$\|\nabla\phi(\bx_k)\| \alpha_k \delta_k \bar c_3\leq\alpha_k \delta_k \bar c_3+\|\nabla\phi(\bx_k)\|^2\alpha_k \delta_k \bar c_3,$$
and putting it in~\eqref{pr27} we get 
\begin{equation} \label{pr28}
    \mbE(\phi(\bx_{k+1})  | \mF_k) \leq \phi(\bx_k)
    -\alpha_k \|\nabla\phi(\bx_k)\|^2 \left( \frac{1}{2L} -v_k \right)
    +\alpha_k^2 \bar c_2+\alpha_k \delta_k \bar c_3,
\end{equation}
where $v_k=\alpha_k \delta_k \bar c_3$. Using Assumption~\ref{ass:sequences} we have that $v_k \to 0$ and thus there exists $\hat{k} \geq \tilde{k}$ such that for all $k \geq \hat{k}$ we have $v_k-1/(2L) \leq -1/(4L)$. Thus, for all $k \geq \hat{k}$ we have 
\begin{equation} \label{pr29}
    \mbE(\phi(\bx_{k+1}) -\phi^* | \mF_k) \leq \phi(\bx_k)-\phi^* 
    -\alpha_k \frac{1}{4L}\|\nabla\phi(\bx_k)\|^2 
    +\alpha_k^2 \bar c_2+\alpha_k \delta_k \bar c_3,
\end{equation}  
where $\phi^*$ is the lower bound of $\phi$ in Assumption~\ref{ass:phi_properties}. Applying Theorem~\ref{thm:rand_variable_conv} with $ U_k=\phi(\bx_k)-\phi^*, \beta_k=0, \xi_k=\alpha_k^2 \bar c_2+\alpha_k \delta_k \bar c_3$ and $\rho_k=\alpha_k \frac{1}{4L}\|\nabla\phi(\bx_k)\|^2$, we conclude that $\phi(\bx_k)$ converges a.s. to a value in $[\phi^*,\,+\infty)$ and 
$$ \sum_{k=\hat{k}}^{\infty} \alpha_k \|\nabla\phi(\bx_k)\|^2 < \infty \quad \mbox{a.s.} $$
Since $\sum_{k=0}^{\infty} \alpha_k = \infty$, the statement~\eqref{pr30} holds.
\end{proof}
    
{
\begin{theorem} \label{th:main2}
Let Assumptions~\ref{ass:phi_properties}-\ref{ass:sequences} hold, $ \{\bx_k\} $ be a sequence generated by Algorithm~\ref{alg:LSOS} and $ K_f\leq K_{max} $.
Then
\begin{equation} \label{pr31}
    \lim_{k \to \infty}\|\nabla\phi(\bx_k)\|=0 \quad \mbox{a.s.}
\end{equation}
and each limit point of $ \{\bx_k\} $ is stationary a.s.~for problem~\eqref{eqn:finite_sum}.
\end{theorem}
}
\begin{proof}
Given that  $K_f \leq K_{max}$, the SA step length is never used. 
Then there exists $\bar{k}$ such that for all $k \geq \bar{k}$ we have
$$ f_{\mD_k}(\bx_{k+1}) \leq   f_{\mD_k}(\bx_{k}) - c_{\min} \|\bg_{\mD_k}(\bx_{k})\|^2 + C_{\max} \, \zeta_k . $$
Let us denote $\mF_{k+1/2}$ the $\sigma$-algebra identified by $\mN_0,\mD_0, \ldots,\mN_{k-1}, \mD_{k-1}, \mN_k$. Then
$$ \mbE(f_{\mD_k}(\bx_{k+1}) | \mF_{k+1/2})  \leq   \mbE(f_{\mD_k}(\bx_{k}) | \mF_{k+1/2})  - c_{\min}\mbE( \|\bg_{\mD_k}(\bx_{k})\|^2 | \mF_{k+1/2}) + C_{\max} \, \zeta_k . $$
% Since $ \mD_k $ is drawn randomly and uniformly from $ \mN $, and the iterates $ \bx_{k} $
% and $ \bx_{k+1}$ are $\mF_{k+1/2}$-measurable, we have
Furthermore,
\begin{equation} \label{rem1}
    \mbE(f_{\mD_k}(\bx_{k+1}) | \mF_{k+1/2})=\phi(\bx_{k+1}), \quad \mbE(f_{\mD_k}(\bx_{k}) | \mF_{k+1/2})=\phi(\bx_{k}), 
\end{equation}
and 
\begin{equation} \label{rem2}
    \mbE( \bg_{\mD_k}(\bx_{k}) | \mF_{k+1/2}) = \nabla \phi(\bx_{k}). 
\end{equation}
The latter equality implies 
\begin{align*}
    \|\nabla \phi(\bx_{k})\|^2 & = \|\mbE( \bg_{\mD_k}(\bx_{k}) | \mF_{k+1/2})\|^2\leq \mbE^2( \|\bg_{\mD_k}(\bx_{k}) \| | \mF_{k+1/2}) \\
    & \leq \mbE( \|\bg_{\mD_k}(\bx_{k}) \|^2 | \mF_{k+1/2}),
\end{align*}
and hence
$$ -c_{\min}\mbE( \|\bg_{\mD_k}(\bx_{k})\|^2 | \mF_{k+1/2})\leq -c_{\min} \|\nabla \phi(\bx_{k})\|^2 . $$
We can conclude that % for all $k \geq \bar{k}$ 
$$ \phi(\bx_{k+1})  \leq   \phi(\bx_{k}) - c_{\min}\|\nabla \phi(\bx_{k})\|^2 + C_{\max} \, \zeta_k, $$
and thus, for any $p \in \mathbb{N}$,
%     $$ \phi(\bx_{\bar{k}+p})
%        \leq \phi(\bx_{\bar{k}})-c_{min}\,\sum_{j=0}^{p-1} \|\nabla \phi(\bx_{\bar{k}+j})\|^2+C_{max} \sum_{j=0}^{p-1} \zeta_{\bar{k}+j}. $$
$$ \phi(\bx_{k+p})
   \leq \phi(\bx_{k})-c_{min}\,\sum_{j=0}^{p-1} \|\nabla \phi(\bx_{k+j})\|^2+C_{max} \sum_{j=0}^{p-1} \zeta_{k+j}. $$
% The boundedness of $\phi$ implies the existence of a constant  $\bar c_4$ such that $\phi(\bx_k) \geq \bar c_4$
% \marconote[inline]{$\bar c_4$ can be replaced by $\phi^*$}  for all $k$.
{Notice that in this case, when $K_f \leq K_{max}$ for all $k=0,1,...$, we have either
$$ f_{\mD_k}(\bx_{k+1}) \leq   f_{\mD_k}(\bx_{k}) - c_{\min} \|\bg_{\mD_k}(\bx_{k})\|^2 + C_{\max} \, \zeta_k \leq f_{\mD_k}(\bx_{k}) + C_{\max} \, \zeta_k $$
or $\bx_{k+1}=\bx_{k}$, and thus for all $k=0,1,...$ there holds
$$f_{\mD_k}(\bx_{k+1}) =   f_{\mD_k}(\bx_{k}) \leq f_{\mD_k}(\bx_{k}) + C_{\max} \, \zeta_k.$$
Applying the conditional expectation $\mbE(\cdot | \mF_{k+1/2})$ we obtain that the following inequality holds for all $k=0,1,...$:
$$ \phi(\bx_{k+1})  \leq   \phi(\bx_{k}) + C_{\max} \, \zeta_k$$
and thus the summability of $\zeta_k$ implies that for all $k=0,1,...$ we have 
$$ \phi(\bx_{k+1})  \leq   \phi(\bx_{0}) + C_{\max} \, \sum_{j=0}^{k} \zeta_j \leq  \phi(\bx_{0}) + C_{\max} \, \sum_{j=0}^{\infty} \zeta_j :=\bar{c}_4 < \infty $$
i.e.,  $\phi(\bx_{k})$ is uniformly bounded from above with a constant $\bar{c}_4$ that depends on $C_{max}$, $\zeta_k$ and $\bx_0$. Therefore, we conclude that for any $p \in \mathbb{N}$,
$$ \phi(\bx_{\bar{k}+p})
    \leq \bar{c}_4-c_{min}\,\sum_{j=0}^{p-1} \|\nabla \phi(\bx_{\bar{k}+j})\|^2+C_{max} \sum_{j=0}^{p-1} \zeta_{\bar{k}+j}.
$$}
\noindent
Since $\phi$ is bounded from below, by taking the expectation, letting $p$ tend to $\infty$ and using the summability of $\zeta_k$, we get
$$ \sum_{k=0}^{\infty} \mbE(\|\nabla \phi(\bx_{k})\|^2)< \infty . $$
Finally, by Markov's inequality we have that for any $ \epsilon>0 $
$$ P(\|\nabla \phi(\bx_{k})\| \geq \epsilon)\leq \frac{\mbE(\|\nabla \phi(\bx_{k})\|^2)}{\epsilon^2} $$
and therefore
$$ \sum_{k=0}^{\infty} P(\|\nabla \phi(\bx_{k})\| \geq \epsilon)< \infty. $$
Finally, Borel-Cantelli Lemma~\cite[Theorem~2.7]{klenke:2014} implies that $ \lim_{k \to \infty} \nabla \phi(\bx_{k})=0 $ a.s., and by the continuity of $\nabla \phi$ we conclude that every limit point of $\{\bx_k\}$ is stationary for $\phi$ a.s..
\end{proof}

{\noindent
The overall convergence statement is a simple combination of the previous two theorems as stated below.

\begin{theorem} \label{th:main}
Let Assumptions~\ref{ass:phi_properties}-\ref{ass:sequences} hold, $ \{\bx_k\} $ be a sequence generated by Algorithm~\ref{alg:LSOS}. %, and $ \delta_k \to 0 $.
Then
\begin{equation*} %\label{pr30}
    \liminf_{k \to \infty}\|\nabla\phi(\bx_k)\|=0 \quad \mbox{a.s.} \, . 
\end{equation*}
Furthermore, if $ K_f\leq K_{max} $ then 
\begin{equation*} %\label{pr31}
    \lim_{k \to \infty}\|\nabla\phi(\bx_k)\|=0 \quad \mbox{a.s.}
\end{equation*}
and each limit point of $ \{\bx_k\} $ is stationary a.s.~for problem~\eqref{eqn:finite_sum}.
\end{theorem}}

\smallskip  
If $ \phi(x) $ is $\mu$-strongly convex we have a stronger convergence result. In this case problem~\eqref{eqn:finite_sum} has a unique solution $\bx^*$.
% Without loss of generality we assume $ \phi(\bx^*) = \phi^* $ and $ L = \max_i L_i $.

\begin{theorem} \label{th:convex}
Let Assumptions~\ref{ass:phi_properties}-\ref{ass:sequences} hold and $ \{\bx_k\} $ be a sequence generated by Algorithm~\ref{alg:LSOS}. %, and $\delta_k \to 0$.
If $ \phi $ is $\mu$-strongly convex, then the sequence $ \{\bx_k\} $ converges a.s. to the unique solution $ \bx^* $ of problem~\eqref{eqn:finite_sum}. 
\end{theorem}
\begin{proof}
First, notice that Assumption~\ref{ass:phi_properties} and strong convexity imply 
\begin{equation} \label{eqn:gradient_error_bound}
    \frac{\mu}{2} \|\bx_k-\bx_*\|^2 \leq \phi(\bx_k) -\phi(\bx_*) \leq \frac{L}{2}\|\nabla \phi(\bx_k)\|^2,
\end{equation}
where $ L $ is as in Assumption~\ref{ass:phi_properties}.
If $K_f \leq K_{max}$, then Theorem~\ref{th:main2} implies 
$$ \lim_{k \to \infty} \bx_k=\bx_* \quad \mbox{a.s.}. $$
On the other hand, if the SA step length is eventually chosen, from {Theorem~\ref{th:main1}} we know that there exists a subsequence $ \{ \bx_k \}_{k \in K \subseteq \mathbb{N}} $ such that 
$$ \lim_{k \in K} \|\nabla \phi(\bx_k)\| = 0 \quad \mbox{a.s.}. $$
This, together with \eqref{eqn:gradient_error_bound}, implies 
$$ \lim_{k \in K} \bx_k=\bx_* \quad \mbox{a.s.}. $$
By the continuity of $\phi$ we get
$$ \lim_{k \in K} \phi(\bx_k)=\phi(\bx_*) \quad \mbox{a.s.}, $$
and according to Theorem~\ref{thm:rand_variable_conv} the whole sequence $ \{ \phi(\bx_k) \} $ converges a.s.. Thus 
$$ \lim_{k \to \infty} \phi(\bx_k)=\lim_{k \in K} \phi(\bx_k)=\phi(\bx_*) \;\; \mbox{a.s.} , $$
which, together with \eqref{eqn:gradient_error_bound}, implies the thesis.
\end{proof}

\begin{remark} \label{rem:infinite_sample}
It is worth noting that the results proved in Theorems~\ref{th:main} and~\ref{th:convex} can be extended to the more general case given by
\begin{equation} \label{eqn:expectation_func}
    \phi(\bx) = \mbE(\psi(\bx;\bxi)),
\end{equation}
where $\bxi\in\Re^m$ is a random vector defined on a probability space. This formulation is usually considered when dealing with infinite samples. % $\left(\Omega,\mF,P\right)$.
In this case, one can approximate the objective function and its derivatives with sample means of the form~\eqref{subsample}, in which $\phi_i(\bx) = \psi(\bx;\bxi_i)$, where $\bxi_i$ is a realization of the random vector $\bxi$. In this setting, the same convergence results hold provided that~\eqref{rem1}-\eqref{rem2} hold and all the functions $\phi_i$ are bounded from below.
\end{remark}

\section{An L-BFGS version of LSOS\label{sec:lbfgs-saga}} 

Subsampling is a natural way of generating approximations of the objective function and its derivatives in the case of finite-sum problems. According to Algorithm~\ref{alg:LSOS}, we have that $ f_{\mN_k}(\bx_k) $, $ f_{\mD_k}(\bx_{k+1}) $ and $ \bg_{\mD_k}(\bx_{k+1}) $ are unbiased estimators of $\phi(\bx_k)$, $ \phi(\bx_{k+1}) $ and $\nabla \phi(\bx_{k+1}) $, respectively.
The derivative estimate corresponding to the sample $ \mN_k $ can be replaced by a more sophisticated one, with the aim, e.g., of improving the performance of the method. The Hessian approximation only needs to have eigenvalues that are uniformly bounded from above and away from zero in order to prove the results presented in the previous section. Therefore, our convergence theory still holds if one replaces the subsampled Hessian approximation with suitable quasi-Newton approximations.

In the case of strongly convex problems, Byrd et al.~\cite{byrd:2016} propose to use subsampled gradients and an approximation of the inverse of the Hessian $\nabla^2 \phi(\bx)$, say $ H_k $, built by means of a stochastic variant of the limited-memory BFGS (L-BFGS) method. Given a memory parameter $m$, $ H_k $ is defined by applying $m$ BFGS updates to an initial matrix, using the $m$ most recent correction pairs $ (\bs_j, \by_j) \in \Re^n\times\Re^n $, like in the deterministic version of the L-BFGS method. The pairs are obtained by averaging iterates, i.e., every $l$ steps the following vectors are computed
\begin{equation}\label{eqn:BHNS_update_1}
    \bw_j = \frac{1}{l}\sum_{i = k-l+1}^{k} \bx_i , \quad \bw_{j-1} = \frac{1}{l}\sum_{i = k-2l+1}^{k-l} \bx_i,
\end{equation}
where $j = k/l$, and they are used to build $\bs_j$ and $\by_j$ as specified next:
\begin{equation}\label{eqn:BHNS_update_2}
    \bs_j = \bw_j - \bw_{j-1}, \quad \by_j = B_{\mT_j}(\bw_j) \,\bs_j,
\end{equation}
where $ \mT_j \subset \mN $. % has cardinality $ T_j = |\mT_j|$.
By defining the set of the $m$ most recent correction pairs as
$$ \left\lbrace (\bs_j, \by_j), \; j = 1,\ldots,m \right\rbrace, $$
the inverse Hessian approximation is computed as
\begin{equation} \label{eqn:BHNS_inverse_hessian}
    H_k  =  H_k^{(m)},
\end{equation}
where for $ j = 1,\ldots,m $
\begin{equation} \label{eqn:BHNS_inverse_hessian_2}
    H_k^{(j)} = \left(I- \frac{\bs_j\, \by_j^\top}{\bs_j^\top\by_j} \right)^\top H_k^{(j-1)}
    \left(I- \frac{\by_j\, \bs_j^\top}{\bs_j^\top\by_j} \right) + \frac{\bs_j\, \bs_j^\top}{ \bs_j^\top \by_j},
\end{equation}
and $ H_k^{(0)} = ( \bs_m^\top\by_m / \|\by_m\|^2 ) \, I $. It can be proved (see \cite[Lemma~3.1]{byrd:2016} and \cite[Lemma~4]{moritz:2016}) that for approximate inverse Hessians of the form \eqref{eqn:BHNS_inverse_hessian} there exist constants $\lambda_1$ and $\lambda_2$, with $\lambda_2 \geq \lambda_1 > 0$, such that
\begin{equation} \label{eqn:BHNS_bounds}
    \lambda_1 I \preceq H_k \preceq \lambda_2 I,
\end{equation}
and hence Assumption~\ref{ass:B(xk)} holds with $\mu = 1 / \lambda_2 $ and $L = 1 / \lambda_1 $.
The authors of~\cite{byrd:2016} propose a stochastic second-order method for convex problems in which the direction is computed as
$$ \bd_k = - H_k\,\bg_{\mN_k}(\bx_k), $$
and prove R-linear decrease of the expected value of the error in the function value. 
For the subsampled gradient and the approximation of the Hessian described above we can easily prove that Lemma~\ref{le:independence} holds.

In the nonconvex case there is no guarantee that the vectors $\bs_j$ and $\by_j$ defined in \eqref{eqn:BHNS_update_2} satisfy the condition $\bs_j^\top\by_j>0$, which is needed to preserve the positive definiteness of the inverse Hessian approximation $H_k$. Therefore, we propose to modify the update scheme \eqref{eqn:BHNS_update_1}-\eqref{eqn:BHNS_update_2} by introducing a damping strategy inspired by the work in \cite{wang:2017siopt}. Let
\begin{equation}\label{eqn:damped_update_1}
    \gamma_{j-1} = \max\left\lbrace \frac{\by_{j-1}^\top\by_{j-1}}{\bs_{j-1}^\top\by_{j-1}},\,\delta\right\rbrace\geq\delta,
\end{equation}
with $\delta>0$. We replace $\by_j$ with the vector
\begin{equation}\label{eqn:damped_update_2}
    \bar{\by}_j = \nu_j \by_j + \left(1-\nu_j\right)\gamma_{j-1}\bs_j,
\end{equation}
where the weight $\nu_j$ is determined as follows:
\begin{equation}\label{eqn:damped_update_3}
    \nu_j = \left\lbrace \begin{array}{ll}
        \frac{0.75\,\gamma_{j-1}\,\bs_{j}^\top\bs_j}{\gamma_{j-1}\,\bs_{j}^\top\bs_j - \bs_{j}^\top\by_j}, & \mbox{if } \bs_{j}^\top\by_j< 0.25\,\gamma_{j-1}\,\bs_{j}^\top\bs_j,\\
        1,                                                                                                           & \mbox{otherwise.}
    \end{array} \right.
\end{equation}
Note that Lemma~3.2 and Lemma~3.3 in \cite{wang:2017siopt} guarantee that the L-BFGS matrices defined by \eqref{eqn:BHNS_inverse_hessian_2} and \eqref{eqn:damped_update_1}-\eqref{eqn:damped_update_3} satisfy~\eqref{eqn:BHNS_bounds}, and hence Assumption~\ref{ass:B(xk)}.

Notice that we can replace the subsampled gradient estimate with alternative gradient estimates coming, e.g., from variance-reduction techniques, which have gained much attention in the literature. This is the case of the stochastic L-BFGS algorithm by Moritz et al.~\cite{moritz:2016}, the stochastic block L-BFGS by Gower et al.~\cite{gower:2016}, and the stochastic damped L-BFGS method by Wang et al.~\cite{wang:2017siopt}, where SVRG gradient approximations are used. The first two methods are suited for strongly convex optimization problems. The method in \cite{moritz:2016} computes the same inverse Hessian approximation as in~\cite{byrd:2016}, while the method in \cite{gower:2016} uses an adaptive sketching technique exploiting the action of a subsampled Hessian on a set of random vectors rather than just on a single vector. Both stochastic BFGS algorithms use constant step lengths and have Q-linear rate of convergence of the expected value of the error in the objective function, but the block L-BFGS one appears more efficient than the other in most of the numerical experiments reported in~\cite{gower:2016}. The method in \cite{wang:2017siopt}, designed for nonconvex problems, uses damped L-BFGS updates, which are computed at each iteration by using a difference of gradients (the gradient sample at the previous iteration is used to ensure independence). Also in this case a constant step lenght is considered, and the authors prove that the expected value of the gradient norm (computed among all the iterations) is led below a predefined threshold in a finite number of steps.

Instead of the SVRG approximation, we apply a mini-batch variant of the SAGA algorithm~\cite{defazio:2014saga}, used in~\cite{gower:2020}. Starting from the matrix $ J^0 \in \Re^{n \times N}$ whose columns are defined as $ J_0^{(i)} = \nabla \phi_i(\bx^0)$, at each iteration we compute the gradient approximation as
\begin{equation}\label{eqn:SAGA_1}
    \bg^{\mathrm{SAGA}}_{\mN_k}(\bx_k)=\frac{1}{N_k} \sum_{i\in\mN_k} \left(\nabla \phi_i(\bx_k) - J_k^{(i)}\right) + \frac{1}{N}\sum_{l=1}^{N} J_k^{(l)},
\end{equation}
and, after updating the iterate, we set
\begin{equation}\label{eqn:SAGA_2}
    J_{k+1}^{(i)} =
    \left\{ \begin{array}{cl}
        J_{k}^{(i)}                      & \mbox{if } i \notin \mN_k,\\
        \nabla \phi_i(\bx_{k+1}) & \mbox{if } i \in \mN_k.
    \end{array} \right.
\end{equation}
As in SVRG, the set $ \mN $ is partitioned into a fixed-number $n_b$ of random mini-batches which are used in order. However, SAGA only requires a full gradient computation at the beginning of the algorithm, while SVRG requires a full gradient evaluation every $n_b$ iterations. Note that the SAGA gradient estimate satisfies the first part of Assumption~\ref{ass:gradient} (gradient unbiasedness). Furthermore, by \cite[Lemma~13]{horvath:2019}, it also satisfies the second part of Assumption~\ref{ass:gradient} (variance boundedness) if the iterates are bounded.

The resulting method, named LSOS-BFGS, is reported in Algorithm~\ref{alg:LSOS-BFGS}. Note that the first L-BFGS update pair is available after the first $2 l$ iterations, and following~\cite{byrd:2016} we take $\bd_k = -\bg(\bx_k)$ for the first $2 l$ iterations.

\begin{algorithm}[htbp]
    \caption{LSOS-BFGS\label{alg:LSOS-BFGS}}
    {\small
        \begin{algorithmic}[1]
            \STATE given $ \bx^0 \in \Re^n $, $ \eta, \vartheta \in (0,1) $,
            $ \{\alpha_k\} \subset \Re_{++} $, $ c_{\min}, C_{\max} \in \Re_{+} $, $ K_{\max}  \in \mathbb{N} $, $ n_b, l, m \in \mathbb{N} $
            \STATE $ K_f = 0 $, $ ind = 0 $, $ k = 0 $
            \WHILE {stop criterion not satisfied}
            \STATE compute a random and uniform partition $ \{\mK_0,\mK_1,\ldots,\mK_{n_b-1}\} $ of $ \mN $
            \FOR{$ r = 0, \ldots, n_b-1 $}
            \STATE choose $ \mN_k = \mK_r $ and compute $\bg_k=\bg^{\mathrm{SAGA}}_{\mN_k}(\bx_k) $ as in~\eqref{eqn:SAGA_1}-\eqref{eqn:SAGA_2}
            \IF {$ k < 2 l $}
            \STATE $ \bd_k = - \bg_k $
            \ELSE
            \STATE $ \bd_k = - H_k\,\bg_k $ with $H_k$ defined in \eqref{eqn:BHNS_inverse_hessian}-\eqref{eqn:BHNS_inverse_hessian_2}
            \ENDIF
            \IF { $ ind = 0 $}
            \STATE find the smallest integer $j \ge 0$ such that the step length $ t_k=\beta^j $ satisfies %\\[-7pt]
            $$
            f_{\mN_k}(\bx_k + t_k \bd_k) \leq f_{\mN_k}(\bx_k) + \eta\, t_k\, \bg_k^\top \bd_k + \vartheta^k
            $$
            %\par \vspace*{-5pt}
            \STATE $ \bar{\bx}_{k} = \bx_k + t_k \bd_k $
            \STATE choose $ \mD_k \subset \mN $ randomly and uniformly 
            \IF {$ \displaystyle  f_{\mD_k}(\bar \bx_k) \le  f_{\mD_k}(\bx_{k}) - c_{\min} \|\bg_{\mD_k}(\bx_{k})\|^2 + C_{\max} \, \zeta_k $} 
            \STATE {$ \bx_{k+1} = \bar \bx_k $} 
            \ELSE
            \STATE {$ \bx_{k+1} = \bx_k $} 
            \STATE $ K_f = K_f + 1$ 
            \IF {$ K_f > K_{\max} $}
            \STATE $ ind = 1$
            \ENDIF
            \ENDIF
            \ELSE
            \STATE {$ t_k = \alpha_k $}
            \STATE {$ \bx_{k+1} = \bx_k + t_k \bd_k $}

            \ENDIF
            \STATE {$ k = k + 1 $}
            \IF{$ \!\!\!\! \mod(k,l)=0 \mbox{ and } k \geq 2 l $}
            \STATE update the L-BFGS correction pairs by using \eqref{eqn:BHNS_update_1}, and \eqref{eqn:BHNS_update_2} or \eqref{eqn:damped_update_1}-\eqref{eqn:damped_update_3}
            \ENDIF
            \ENDFOR
            \ENDWHILE
        \end{algorithmic}
    }
\end{algorithm}

\begin{remark}
    By assuming that all the functions $\phi_i$ have Lipschitz-continuous gradients with Lipschitz constant bounded from above by $L$, we have that the gradient estimate
    % $\bg_{\mN_k}(\bx)$ and
    $\bg^{\mathrm{SAGA}}_{\mN_k}(\bx)$ is Lipschitz continuous with Lipschitz constant bounded from above by $L$.
\end{remark}

\section{Numerical experiments\label{sec:experiments}}

We developed a MATLAB implementation of Algorithm~LSOS-BFGS and tested it on both nonconvex and convex finite-sum problems arising in machine learning.

The nonconvex problems we considered are nonlinear least-squares problems {as the ones in \cite{xu:2020}}. Given $N$ pairs $ (\ba_i, b_i) $, where $\ba_i \in \Re^n$ is a data point and $b_i \in \{0,1\}$ the corresponding response, and considering a sigmoidal kernel, one can obtain a problem of the form~\eqref{eqn:finite_sum}, where
$$
\phi_i(\bx)=\frac{1}{2}\left(b_i - \frac{1}{1 + e^{-\ba_i^\top\bx}}\right)^2.
$$
By setting $u_i(\bx)=(1 + e^{-\ba_i^\top\bx})^{-1}$, the gradient and the Hessian of $\phi_i$ can be written as
\begin{eqnarray*}
    %    &\nabla \phi_i(\bx)=-\left( u_i(\bx)\cdot\left(1-u_i(\bx)\right)\cdot\left(b_i-u_i(\bx)\right)\right)\ba_i,\\
    %    &\nabla^2 \phi_i(\bx)= -u_i(\bx)\cdot\left(1-u_i(\bx)\right)\cdot\left(b_i- 2\,(1+b_i)\cdot u_i(\bx) + 3\,u_i(\bx)^2 \right) \ba_i{\ba^\top_i}.
    & \nabla \phi_i(\bx)=-\left( u_i(\bx) \left(1-u_i(\bx) \right) \left(b_i-u_i(\bx)\right) \right)\ba_i,\\
    & \nabla^2 \phi_i(\bx)= -u_i(\bx) \left( 1-u_i(\bx) \right) \left( b_i- 2 \, (1+b_i)  u_i(\bx) + 3 \, u_i(\bx)^2 \right) \ba_i{\ba^\top_i}.
\end{eqnarray*}

The convex problems come from the training of a linear classifier by minimizing the $\ell_2$-regularized logistic regression model. Given the pairs $ (\ba_i, b_i) $, where $b_i \in \{-1,1\}$ is the class label associated with the training point $\ba_i \in \Re^n$, an unbiased hyperplane approximately separating the two classes can be found by solving problem~\eqref{eqn:finite_sum}, where
$$
\phi_i(\bx)=\log\left(1+ e^{-b_i\,\ba_i^\top\bx}\right) + \frac{\mu}{2}\|\bx\|^2
$$
and $\mu > 0$. By setting $z_i(\bx)=1+e^{-b_i\,\ba_i^{\top}\bx}$, the gradient and the Hessian of $\phi_i$ can be written as
$$
\nabla \phi_i(\bx)=\frac{1-z_i(\bx)}{z_i(\bx)}b_i\,\ba_i+\mu \bx, \qquad \nabla^2 \phi_i(\bx)=\frac{z_i(\bx)-1}{z^2_i(\bx)}\ba_i{\ba^\top_i}+\mu I.
$$
Note that from $(z_i(\bx)-1) / z_i^2(\bx) \in (0,1)$ it follows that $\phi_i$ is $\mu$-strongly convex and
$$
\mu I \preceq \nabla^2 \phi_i(\bx) \preceq L I, \quad L=\mu+\max_{i=1,...,N}\|a_i\|^2.
$$

We observe that in both cases $\nabla \phi_i(\bx)$ and $\nabla^2 \phi_i(\bx)$ can be computed at low cost after $\phi_i(\bx)$ has been computed, because of the special form of the derivatives and the fact that the dominant cost is usually the scalar product in~$u_i(\bx)$ and $z_i(\bx)$.

All the tests were run with MATLAB R2019b on the \emph{magicbox} server operated by the Department of Mathematics and Physics at the University of Campania ``L.~Vanvitelli". It is equipped with 8 Intel Xeon Platinum 8168 CPUs, 1536 GB of RAM and Linux CentOS 7.5 operating system. A single Intel Xeon CPU with 192 GB of RAM was used in the experiments.

\subsection{Results on nonconvex problems}
To assess the performance of LSOS-BFGS on the solution of nonconvex problems, we compared it with the following algorithms:
\begin{itemize}
    \item the variance-reduced stochastic damped L-BFGS algorithm SdLBFGS-VR \cite{wang:2017siopt};
    \item a mini-batch variant of the SAGA algorithm equipped with the same line search used in LSOS-BFGS, referred to as SAGA.
\end{itemize}

We developed our own MATLAB implementation of SdLBFGS-VR, ensuring consistency with LSOS-BFGS in terms of function and gradient evaluation costs. Since the convergence results in the nonconvex case are stated in terms of gradient norm (see Theorem~\ref{th:main} and the results in \cite{wang:2017siopt}), that value was used as a measure of optimality. The constant step length for SdLBFGS-VR was chosen by means of a grid search over the set $S = \{ 1, 5 \cdot 10^{-1}, 10^{-1}, \ldots, 5 \cdot 10^{-4}, 10^{-4} \} $, selecting the step length that yielded the smallest gradient in a fixed execution time. Moreover, we set the L-BFGS memory equal to $10$ and $\delta=10^{-2}$ in the damping strategy. In Algorithms~LSOS-BFGS and SAGA we set $ \vartheta = 0.999$ and the initial line-search step length equal to~$1$. Concerning the L-BFGS update in LSOS-BFGS, the parameters were set as $ m=10 $ and $ l=5 $, and $\delta=10^{-2}$ in~\eqref{eqn:damped_update_1}. We also set, for both LSOS-BFGS and SAGA, $c_{\min}=10^{-6}$, $C_{\max}=10^2$ and $K_{\max} = 10^5$. Finally, {we considered the following LSOS-BFGS specific quantities: $ \mD_k $ with cardinality 1, the predefined step length $\alpha_k = \frac{1}{\|\bd_0\|} \frac{T}{T+k}$, with $T = 10^6$, the gradient sample size equal to $ N_k =  \lceil\sqrt{N} \rceil $ and the Hessian sample size ($|\mT_j|$ in~\eqref{eqn:BHNS_update_2}) equal to $3\lceil\sqrt{N} \rceil$.} 

The comparison was performed by using binary classification datasets from the LIBSVM collection (\url{https://www.csie.ntu.edu.tw/~cjlin/libsvmtools/datasets/}). The results on the four problems listed in Table~\ref{tab:datasets} are representative of the general behavior of LSOS-BFGS.
\begin{table}[t!]
    {\small
        \caption{Datasets from LIBSVM used in the nonconvex numerical experiments and in the comparison of LSOS-BFGS with GGR, MNJ and SAGA. For each dataset the number of training points and the number of features (space dimension) are reported. Whenever a training set was not specified in LIBSVM, we selected it by using the MATLAB \texttt{crossvalind} function so that it contained 70\% of the available data.\label{tab:datasets}}
        \begin{center}
            \begin{tabular}{|l|r|r|}
                \hline
                \multicolumn{1}{|l|}{name} & \multicolumn{1}{r|}{$N$} & \multicolumn{1}{r|}{$n$} \\ \hline
                \textsf{gisette}           &                     6000 &                     5000 \\
                \textsf{rcv1}              &                    20242 &                    47236 \\ 
                \textsf{real-sim}          &                    50617 &                    20958 \\
                \textsf{w8a}               &                    49749 &                      300 \\
                \hline
            \end{tabular}
        \end{center}
    }
\end{table}

%In the experiments we set the gradient sample size equal to $ \lceil\sqrt{N} \rceil $ and the Hessian sample size for LSOS-BFGS ($\mT_j$ in \eqref{eqn:BHNS_update_2}) equal to $3\lceil\sqrt{N} \rceil$.
The algorithms were stopped when a maximum execution time was reached, i.e., 60 seconds for \textsf{w8a} and \textsf{gisette}, and 300 seconds for \textsf{real-sim} and \textsf{rcv1}. 
Each algorithm was run 20 times on each problem and the average error and average execution time spent until each iteration $k$ were computed. In the plots we represent the mean optimality measures as lines together with shaded regions corresponding to 95\% confidence intervals.

\begin{figure}[h]
    \includegraphics[width=0.47\textwidth]{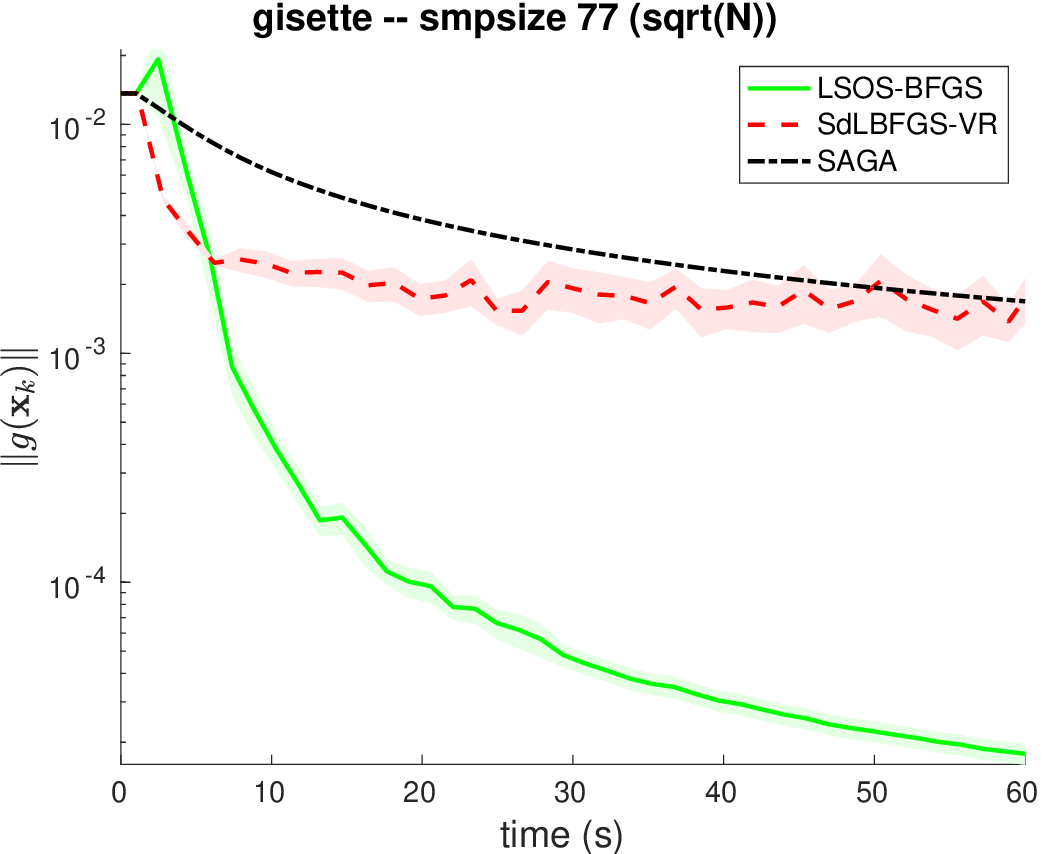}\includegraphics[width=0.47\textwidth]{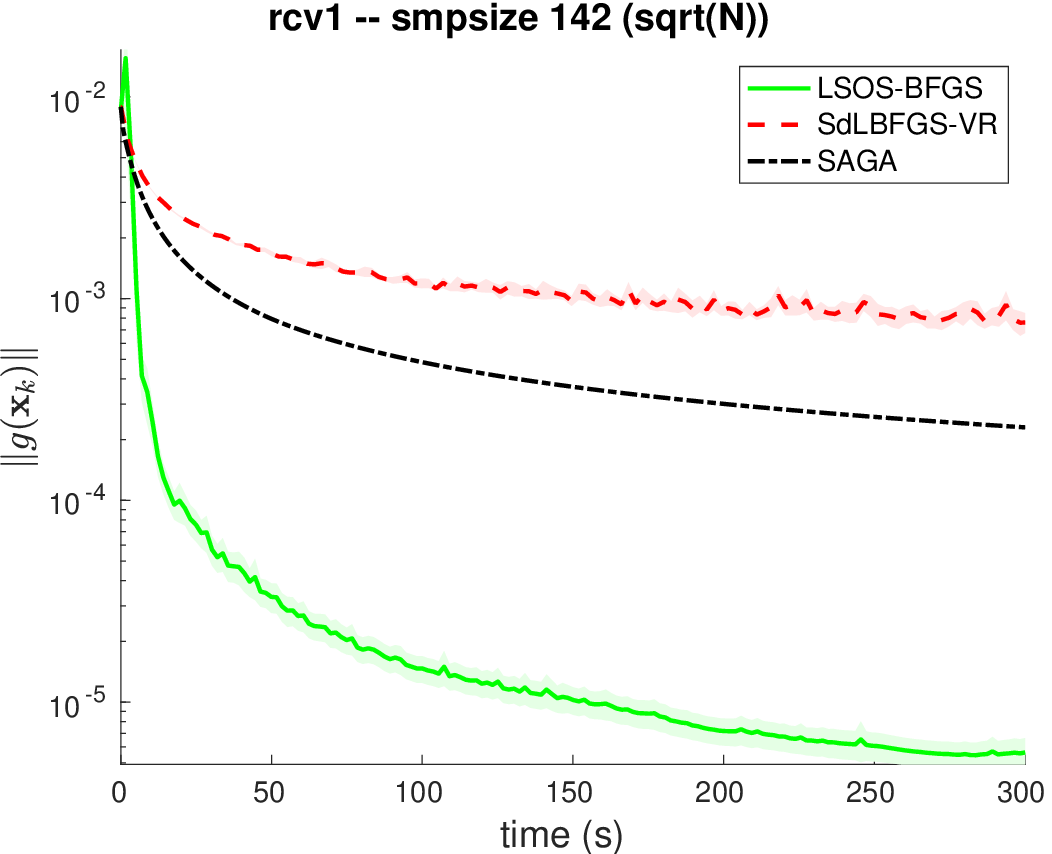} \\
        \includegraphics[width=0.47\textwidth]{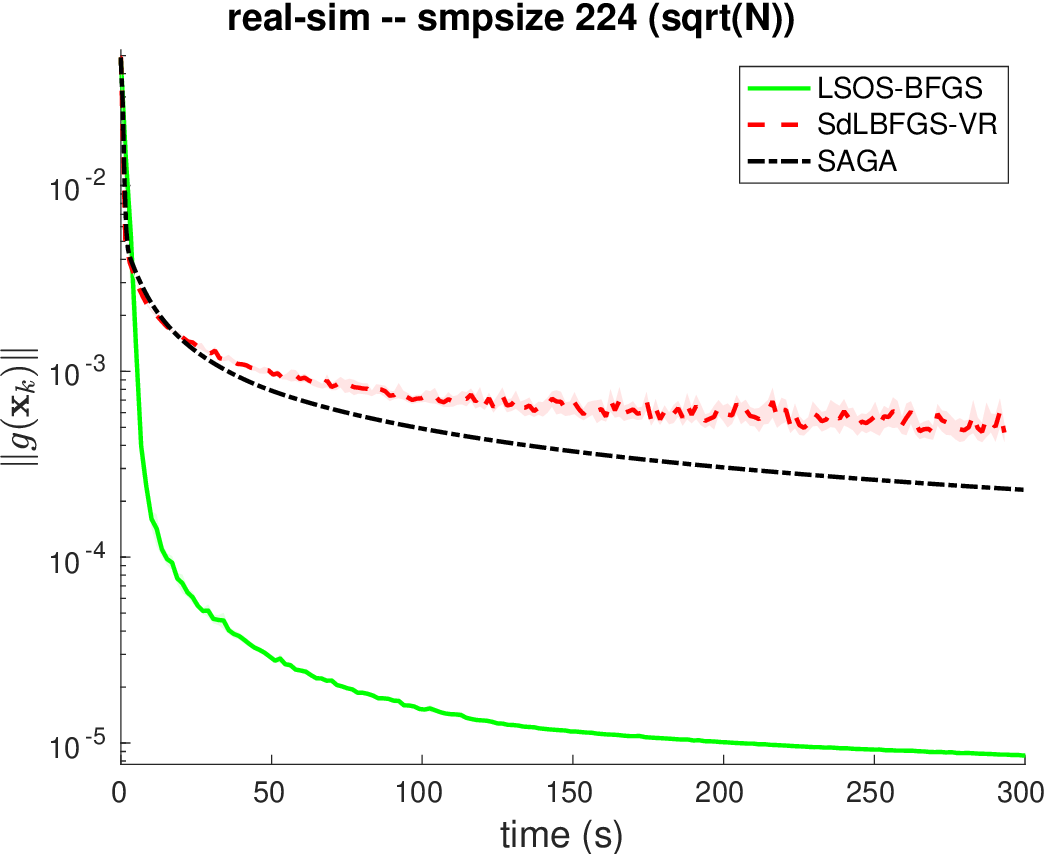}    \includegraphics[width=0.47\textwidth]{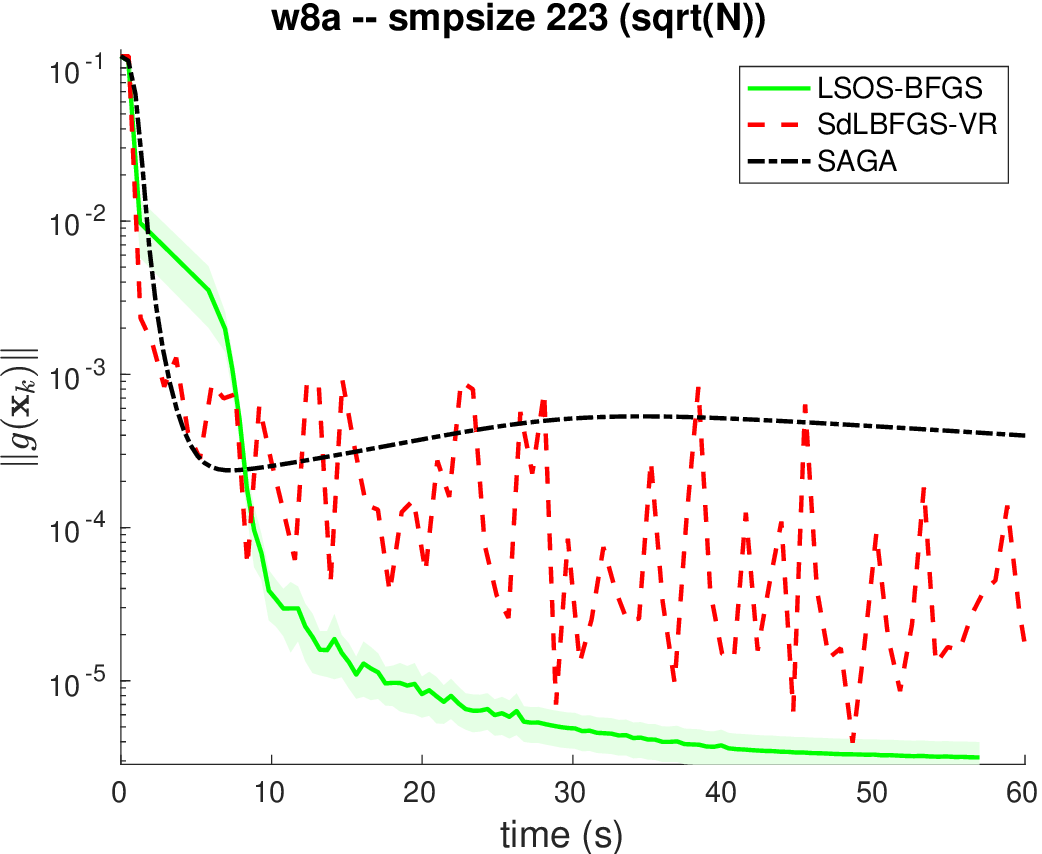}
    \caption{Nonconvex least squares problems: comparison of LSOS-BFGS, SdLBFGS-VR and SAGA in terms of gradient norm versus execution time.\label{fig:ncvx_least_squares}}
\end{figure}

Figure~\ref{fig:ncvx_least_squares} shows a comparison among the three algorithms in terms of the average gradient norm versus the average
execution time. For SdLBFGS-VR, the grid search for the step lengths was performed on the first of the 20 runs and then fixed for the remaining 19 runs. 
The results show that LSOS-BFGS is more effective than the two competitors in reducing the norm of the gradient of the objective function, and hence in converging to stationary points.
It is worth noting that the average number of rejected steps ($K_f$ in Algorithm~LSOS-BFGS) was $0$ for \textsf{gisette}, below 1\% for \textsf{w8a} and \textsf{real-sim}, and around 6\% for \textsf{rcv1}. The maximum number of failures ($K_{\max}$) was never reached, hence LSOS-BFGS never resorted to the use of SA step lengths. 

\subsection{Results on convex problems}
We performed two different sets of experiments for assessing the performance of LSOS-BFGS on the solution of strongly convex problems. In the first one  LSOS-BFGS was compared with the following algorithms:
\begin{itemize}
    \item the stochastic L-BFGS algorithms proposed in~\cite{gower:2016}, referred to as GGR;
    \item the stochastic L-BFGS algorithms proposed in~\cite{moritz:2016}, referred to as MNJ;
    \item the mini-batch variant of the SAGA used in the previous experiments.
\end{itemize}
The implementations of GGR and MNJ were taken from the MATLAB StochBFGS code available from \url{https://perso.telecom-paristech.fr/rgower/software/StochBFGS_dist-0.0.zip} and were used with constant step lengths selected by means of a grid search over the same set $S$ used in the nonconvex case. We chose the step lengths leading to the best results in terms of objective function error versus execution time. On this class of problems we observed  that the performance of LSOS-BFGS improved if the line search was started from a value smaller than~1 (which led to a reduction in the number of line-search steps performed at each iteration). {Thus, we began the line search from a value $ t_{\mathrm{ini}} $ selected by means of a grid search over $ S $, obtaining $ t_{\mathrm{ini}}= 1\cdot10^{-2} $ for \textsf{gisette} and \textsf{w8a}, and $ t_{\mathrm{ini}}=5\cdot10^{-3} $ for \textsf{rcv1} and \textsf{real-sim}. The same strategy was adopted for the line-search version of SAGA, obtaining $ t_{\mathrm{ini}} = 1\cdot10^{-1}$ for \textsf{gisette}, $ t_{\mathrm{ini}} = 1$ for \textsf{rcv1} and \textsf{real-sim}, and $t_{\mathrm{ini}} = 5\cdot10^{-2}$ for \textsf{w8a}. The sample $ {\mathcal D}_k $ at line 15 of LSOS-BFGS algorithm was again of size 1}. Concerning the L-BFGS update, we set $ m=10 $ and $ l=5 $ as in the case of nonconvex problems, and used these values also in the MNJ algorithm. For GGR, following the indications coming from the results in~\cite{gower:2016}, we set $ m=5 $ and used the sketching based on the previous directions (indicated as \texttt{prev} in~\cite{gower:2016}), with sketch size $l =  \left\lceil\sqrt[3]{n}\right\rceil$. 

This comparison was performed by using again the datasets listed in Table~\ref{tab:datasets}. According to the experiments reported in~\cite{gower:2016}, we set the gradient sample size equal to $ \lceil\sqrt{N} \rceil $, the Hessian sample size for LSOS-BFGS and MNJ equal to $3\lceil\sqrt{N} \rceil$, and $ \mu = 1 / N $ as regularization parameter. The same stopping criterion (based on execution time) was used for the nonconvex
problem instances built with the same datasets. For each problem we computed a solution with high accuracy by using
the (deterministic) L-BFGS implementation by Mark Schmidt, available from \url{https://www.cs.ubc.ca/~schmidtm/Software/minFunc.html}.

\begin{figure}[h]
    \includegraphics[width=0.47\textwidth]{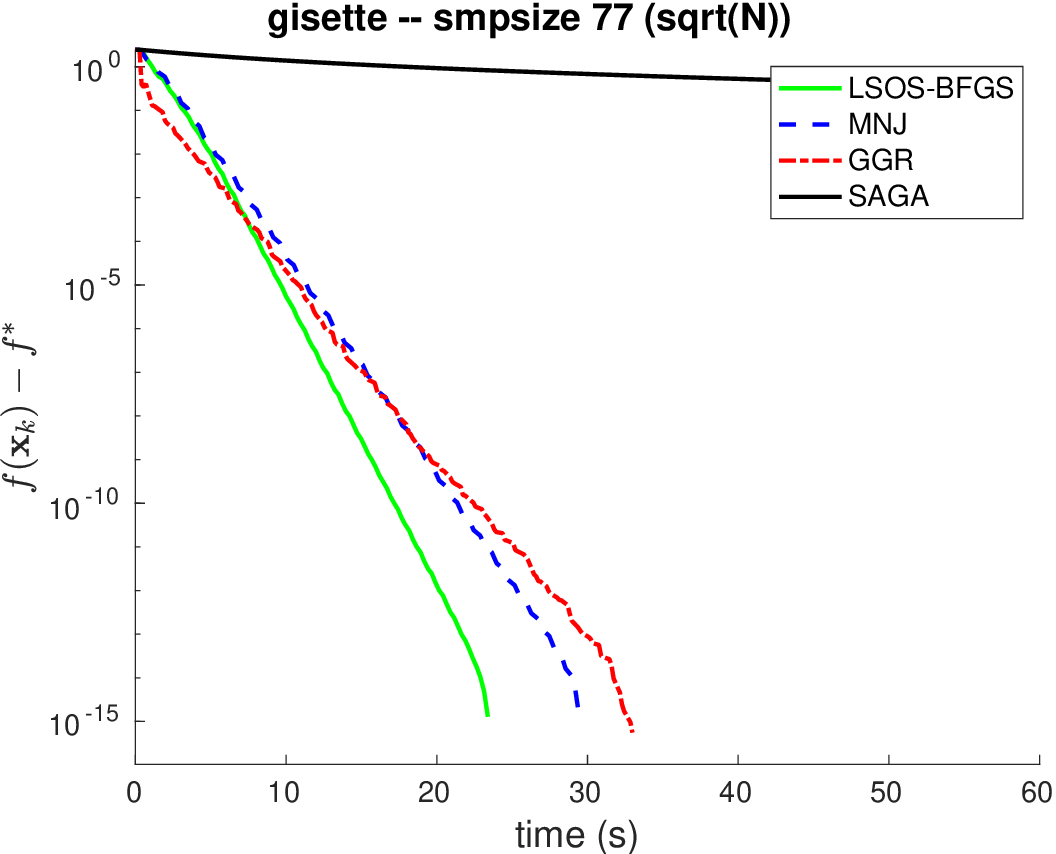}
    \includegraphics[width=0.47\textwidth]{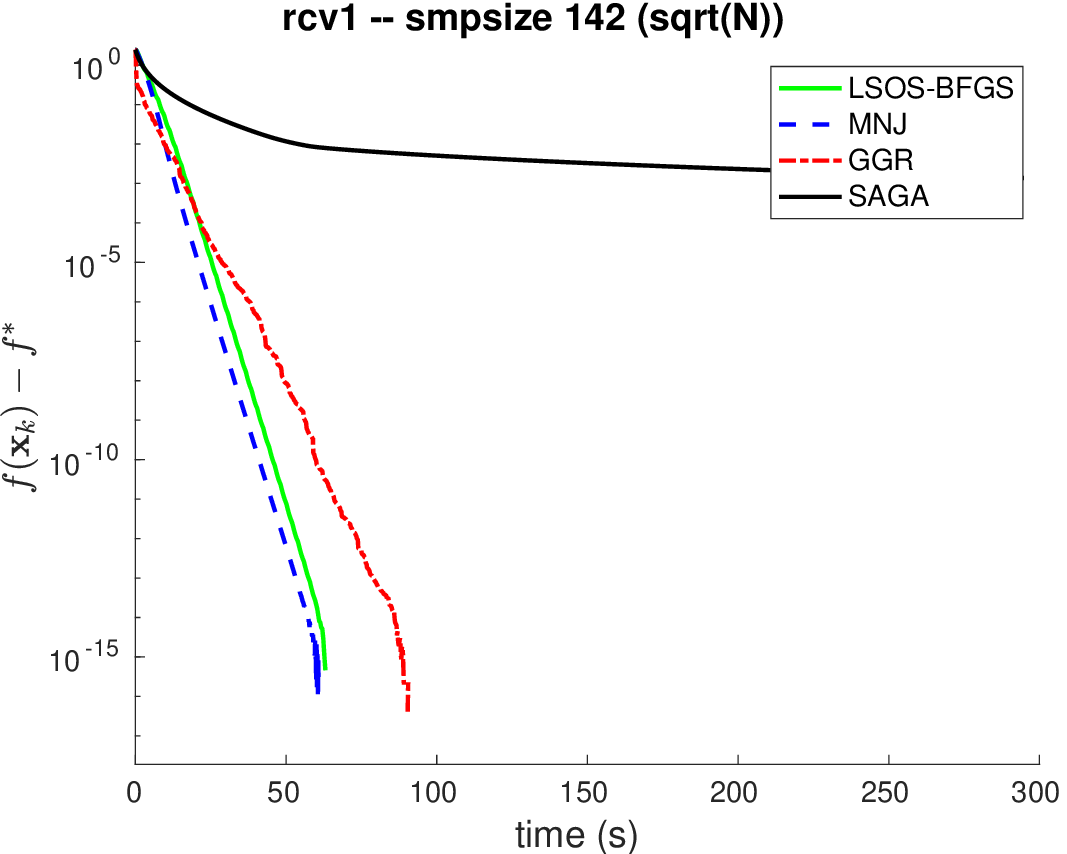}\\[3mm]
    \includegraphics[width=0.47\textwidth]{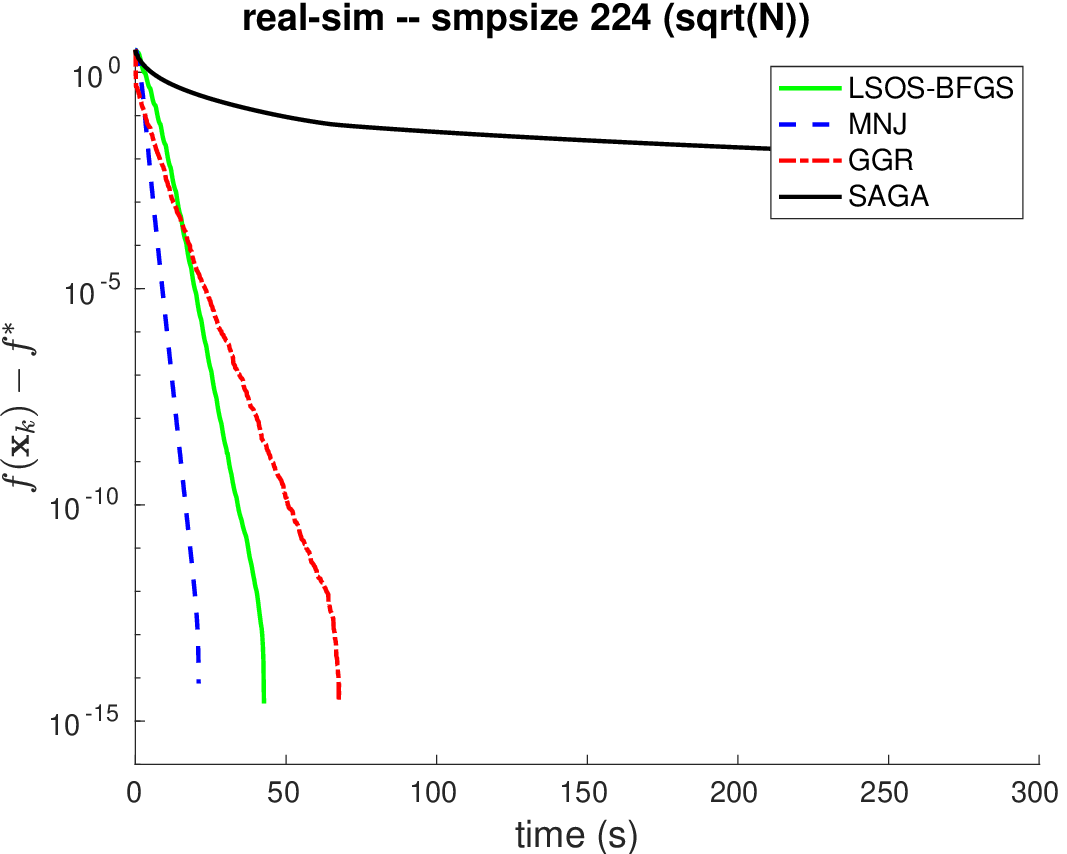}
    \includegraphics[width=0.47\textwidth]{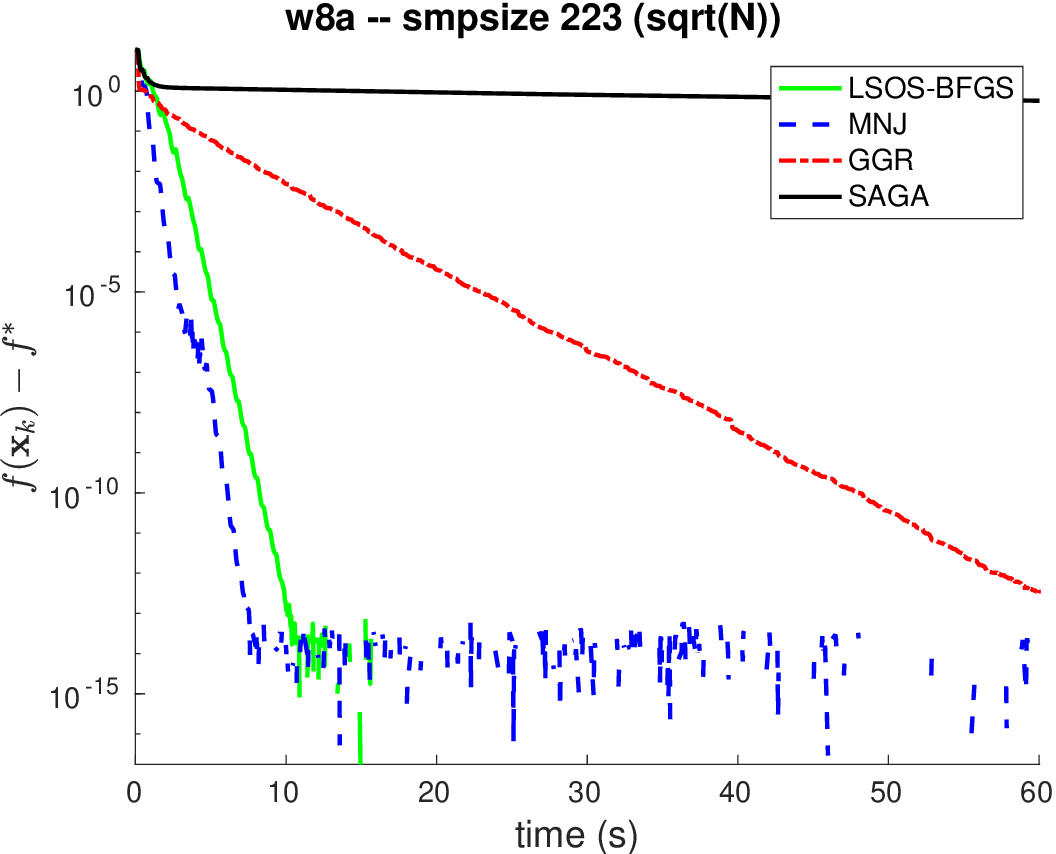}
\caption{Binary classification problems: comparison of LSOS-BFGS, MNJ, GGR and SAGA in terms of function error versus execution time. \label{fig:binary_class}}
\end{figure}

\begin{figure}[h]
    \includegraphics[width=0.47\textwidth]{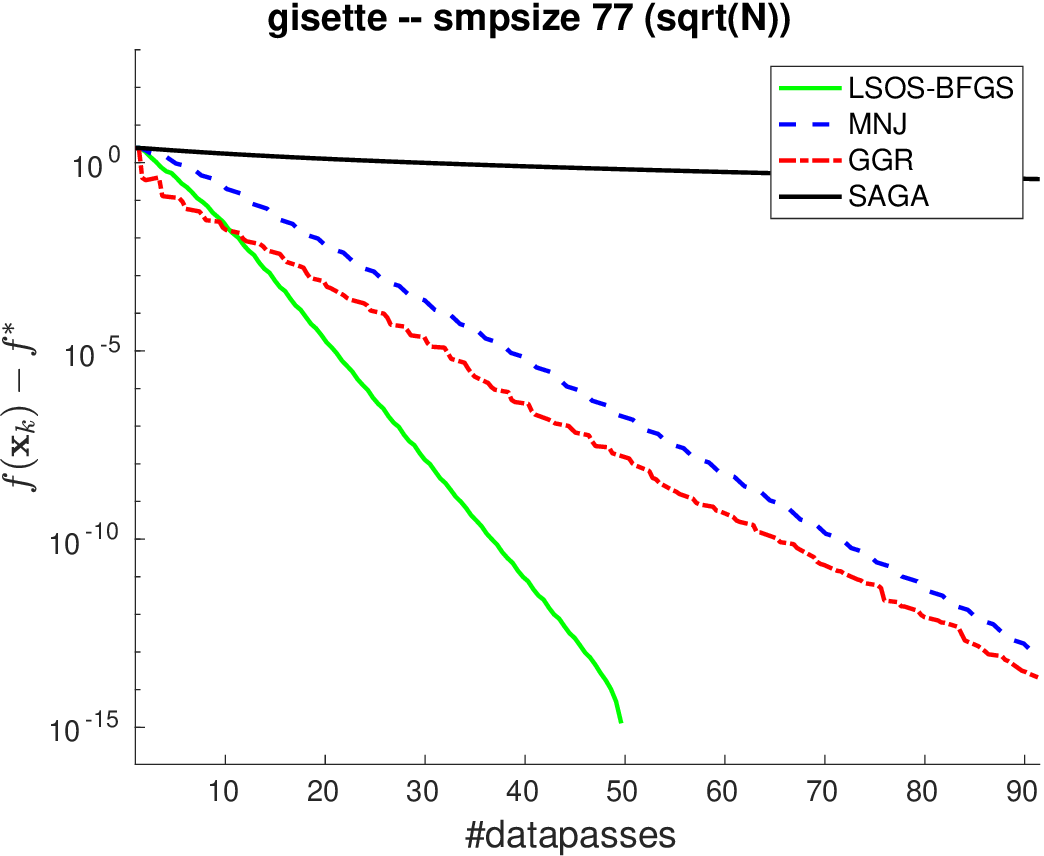}\includegraphics[width=0.47\textwidth]{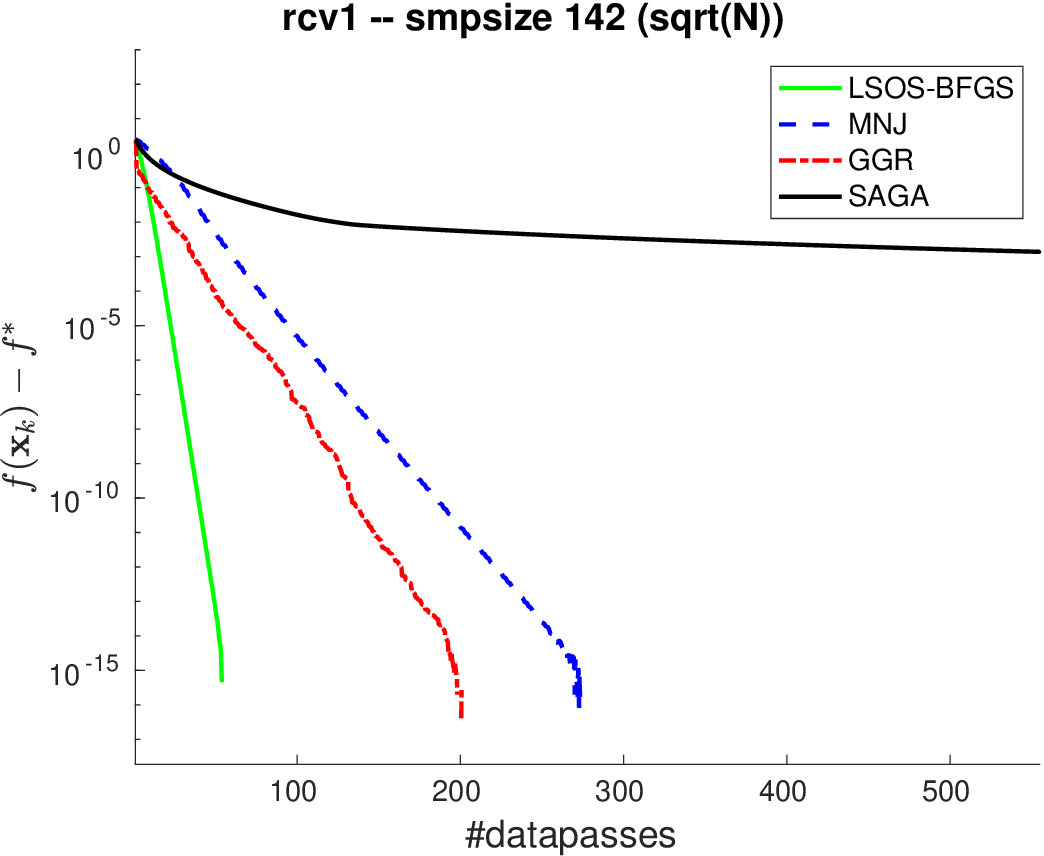} \\[3mm]
    \includegraphics[width=0.47\textwidth]{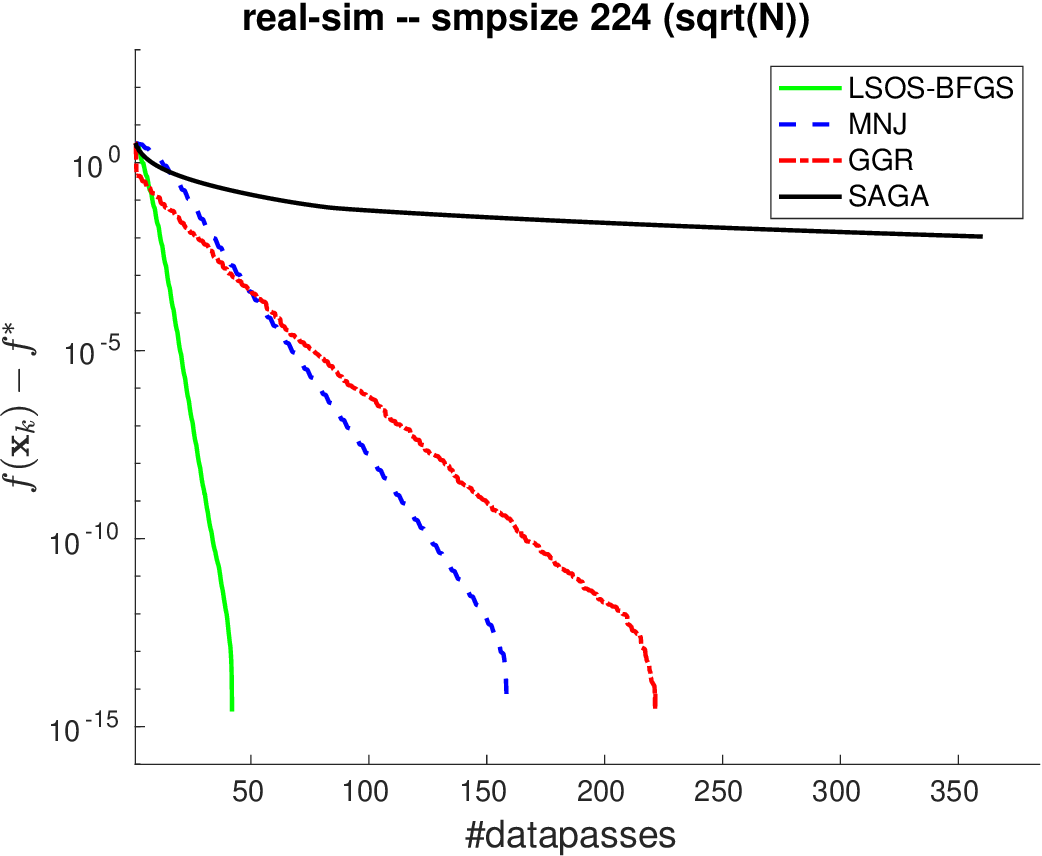}
    \includegraphics[width=0.47\textwidth]{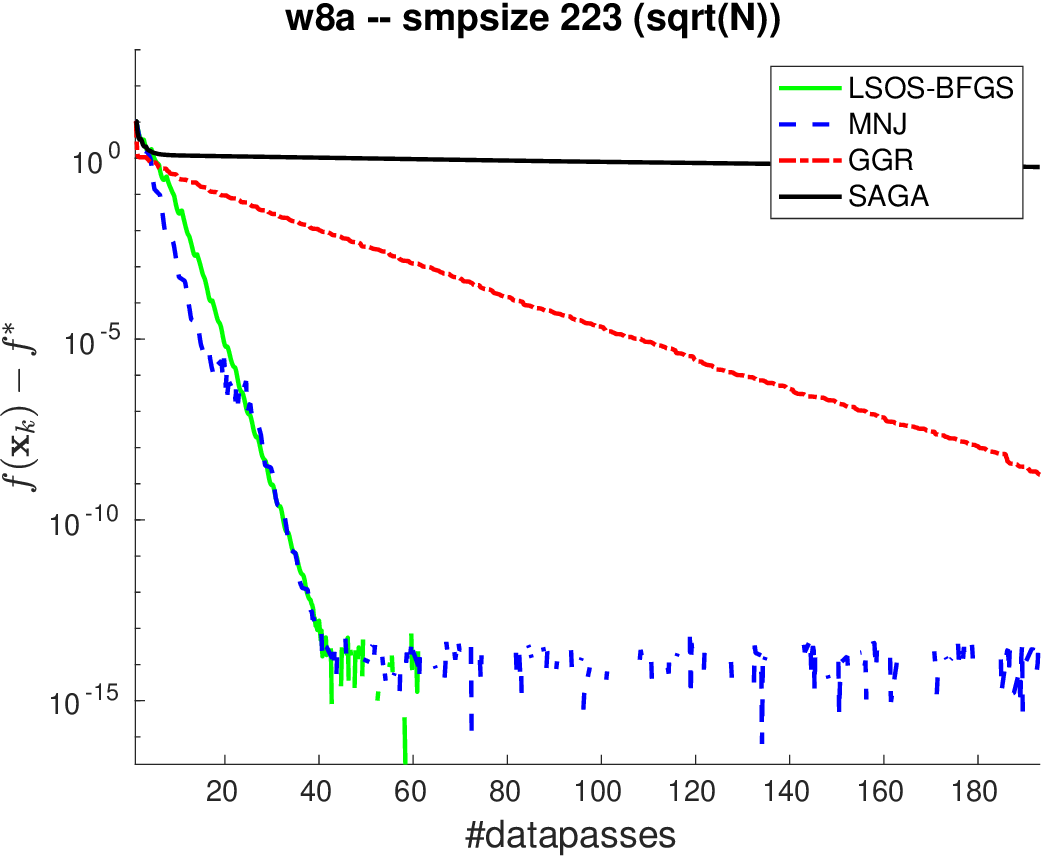}
    \caption{Binary classification problems: comparison of LSOS-BFGS, MNJ, GGR and SAGA in terms of function error versus data passes. \label{fig:binary_class_datapasses}}
\end{figure}

{Figures~\ref{fig:binary_class} and \ref{fig:binary_class_datapasses} show a comparison among the four algorithms in terms of the average absolute error on the
objective function (with respect to the optimal value) versus the average execution time and the number of data passes, respectively}.  As in the previous experiments, the error and the times are averaged over 20 runs and the plots show the 95\% confidence interval with respect to the error. The grid search for defining or initializing the step lengths was performed again on the first of the 20 runs and then fixed for the remaining 19 ones. 

The results {in terms of execution time (Figure~\ref{fig:binary_class})} show that LSOS-BFGS outperforms the other stochastic L-BFGS algorithms
on \textsf{gisette}, and outperforms GGR on \textsf{rcv1},  \textsf{real-sim} and \textsf{w8a}; furthermore, it is comparable with MNJ on \textsf{rcv1}.
%It is worth noting that on \textsf{w8a} the error for GGR tends to increase after a certain iteration, while the other L-BFGS algorithms seem to keep a ``straight'' decrease.
Moreover, LSOS-BFGS is always more efficient than the line-search-based mini-batch SAGA, showing that the introduction of stochastic second-order information
is crucial for the performance of the algorithm.

{We now focus on the comparison in terms of number of data passes, which provides a measure of oracle complexity for the four algorithms. By looking at Figure~\ref{fig:binary_class_datapasses} it is clear how, thanks to the use of line searches and SAGA (which does not need to compute a full gradient at each step), LSOS-BFGS is able to outperform its competitors on all the cases but \textsf{w8a}, on which it is comparable with MNJ.}
We observe that in LSOS-BFGS the average number of rejected steps was $0$ for \textsf{gisette} and below 6\% for \textsf{rcv1}, \textsf{real-sim} and \textsf{w8a}. Again, the maximum number of failures ($K_{\max}$) was never reached, and LSOS-BFGS always determined the step lengths by the nonmonotone line search. 

In the second set of experiments on convex problems we compared LSOS-BFGS with the Incremental Quasi-Newton (IQN) method~\cite{mokhtari:2018}. Note that the latter has proven superlinear convergence rate, but high memory requirements, i.e., multiple Hessian matrices to be stored. We used the MATLAB implementation of IQN available from \url{https://github.com/hiroyuki-kasai/SGDLibrary}, but in our opinion we improved it by avoiding some redundant operations and explicit inversions of matrices. We note that for these experiments we had to use smaller datasets (see Table~\ref{tab:datasets_2}), because of the high memory cost of the IQN implementation, which requires storing a full Hessian matrix for each sample in the training set. Thus, in LSOS-BFGS we set the gradient sample size as {$ N_k = 10$ } and the Hessian sample size equal to {$|\mT_j|= 30$.  } The remaining L-BFGS parameters were chosen as in the previous experiments, {including the cardinality of $ {\mathcal D}_k $ equal to 1}. Moreover, we set to $10^{-1}$ the starting value of the step length in the line search for all the four datasets. Recall that for the superlinear convergence property to hold, IQN had to be used with step length $1$.

\begin{table}[ht!]
    {\small
        \caption{Datasets from LIBSVM used in the comparison of LSOS-BFGS with IQN on convex problems. For each dataset the number of training points and the number of features (space dimension) are reported.\label{tab:datasets_2}}
        \begin{center}
            \begin{tabular}{|l|r|r|}
                \hline
                \multicolumn{1}{|l|}{name} & \multicolumn{1}{r|}{$N$} & \multicolumn{1}{r|}{$n$} \\ \hline
                a9a       &  32561 &      123 \\
                cod-rna   &  59535 &        8 \\
                mushrooms &   8124 &      112 \\
                phishing  &  11055 &       68 \\ \hline
            \end{tabular}
        \end{center}
    }
\end{table}
\begin{figure}[h!]
    \includegraphics[width=0.47\textwidth]{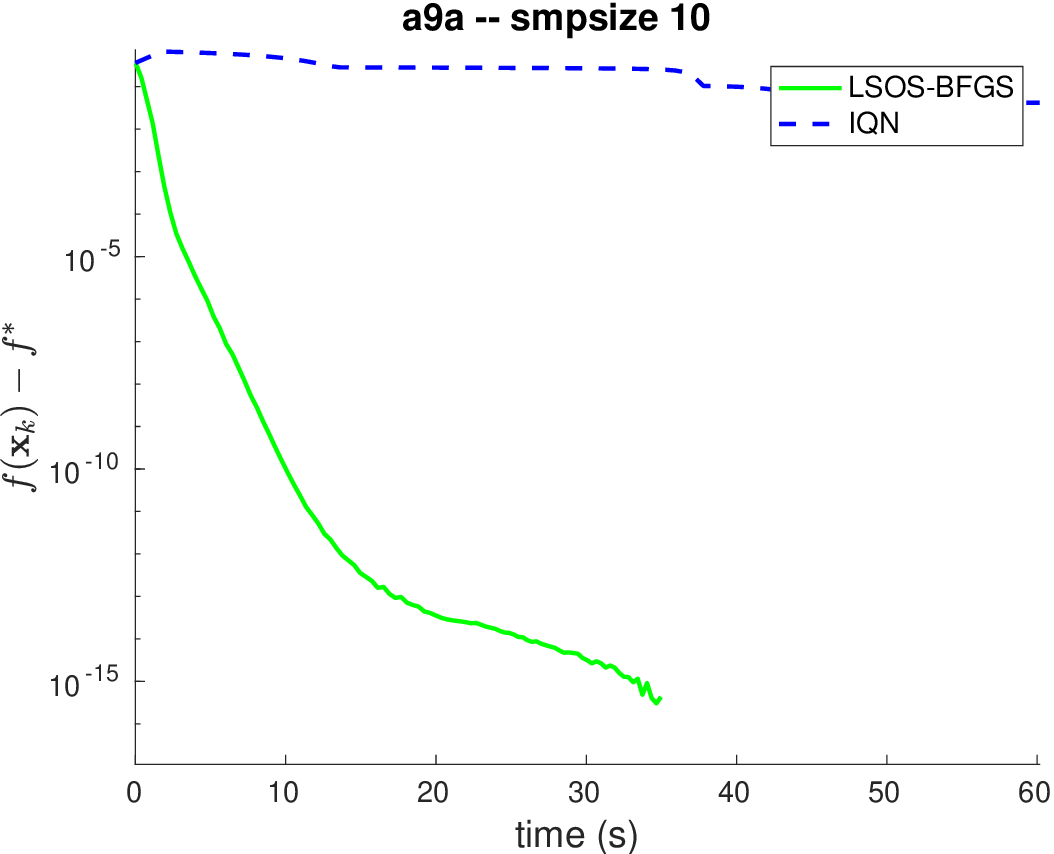}
    \includegraphics[width=0.47\textwidth]{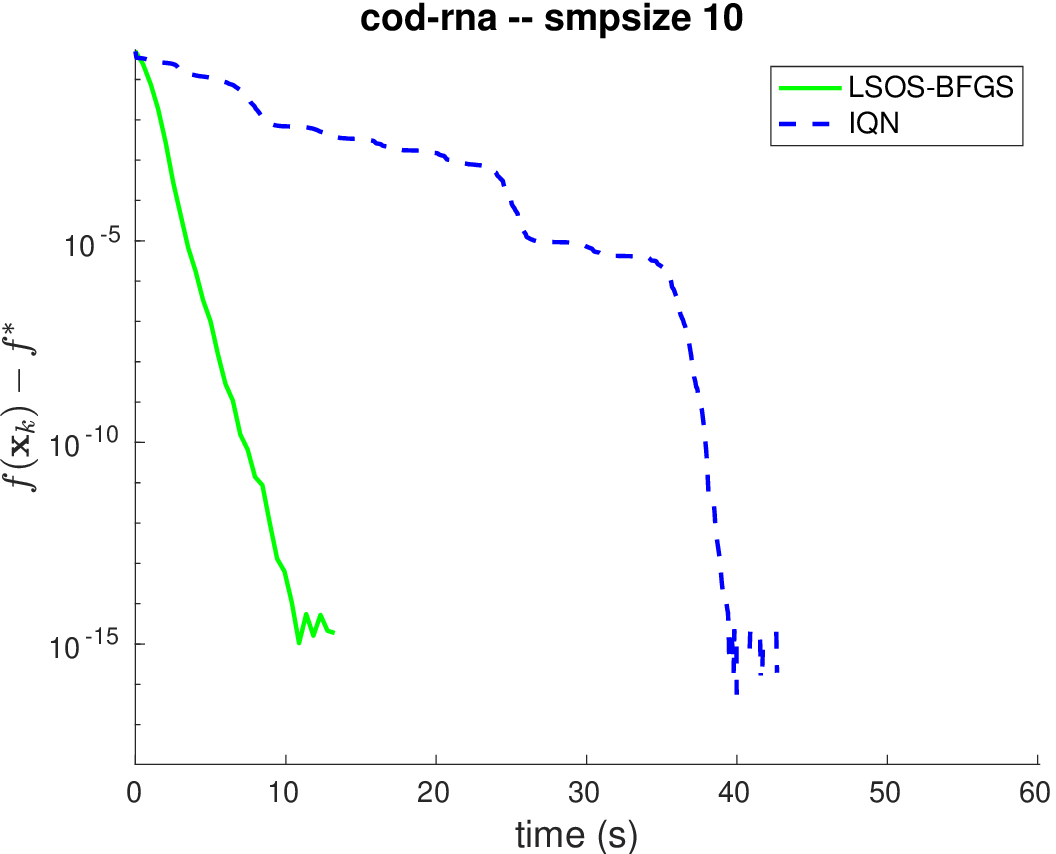}\\
    
    \includegraphics[width=0.47\textwidth]{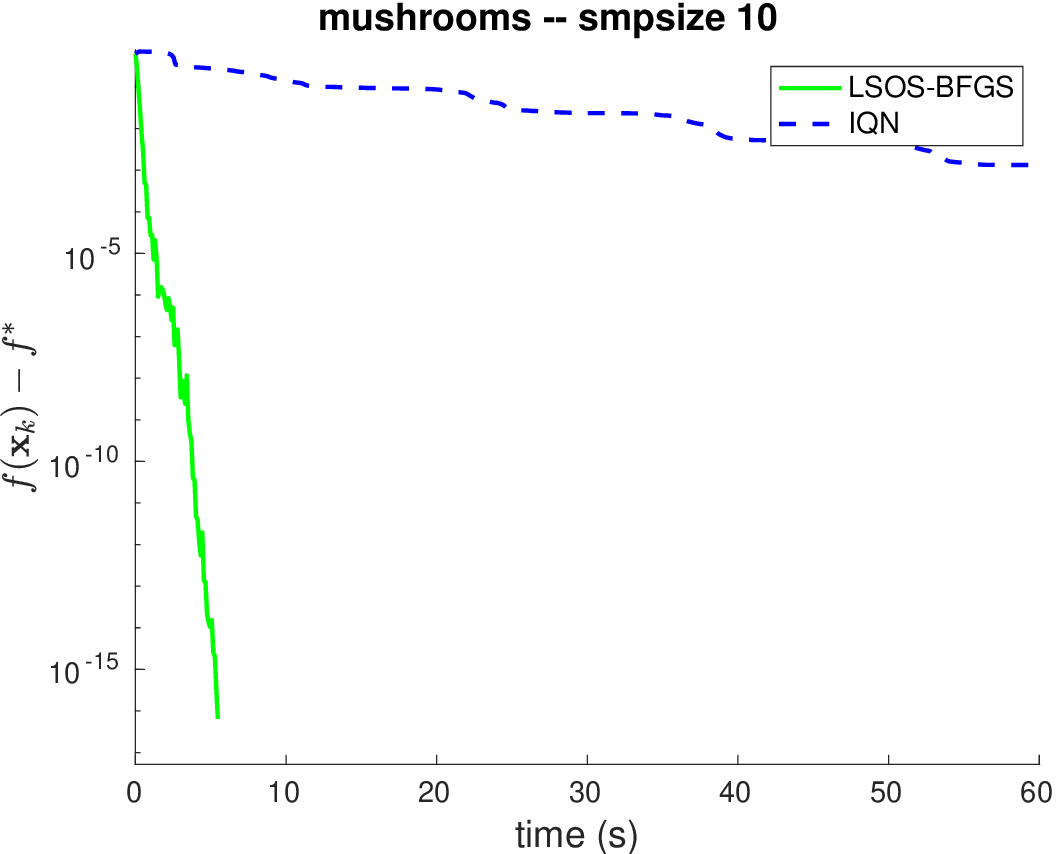}
    \includegraphics[width=0.47\textwidth]{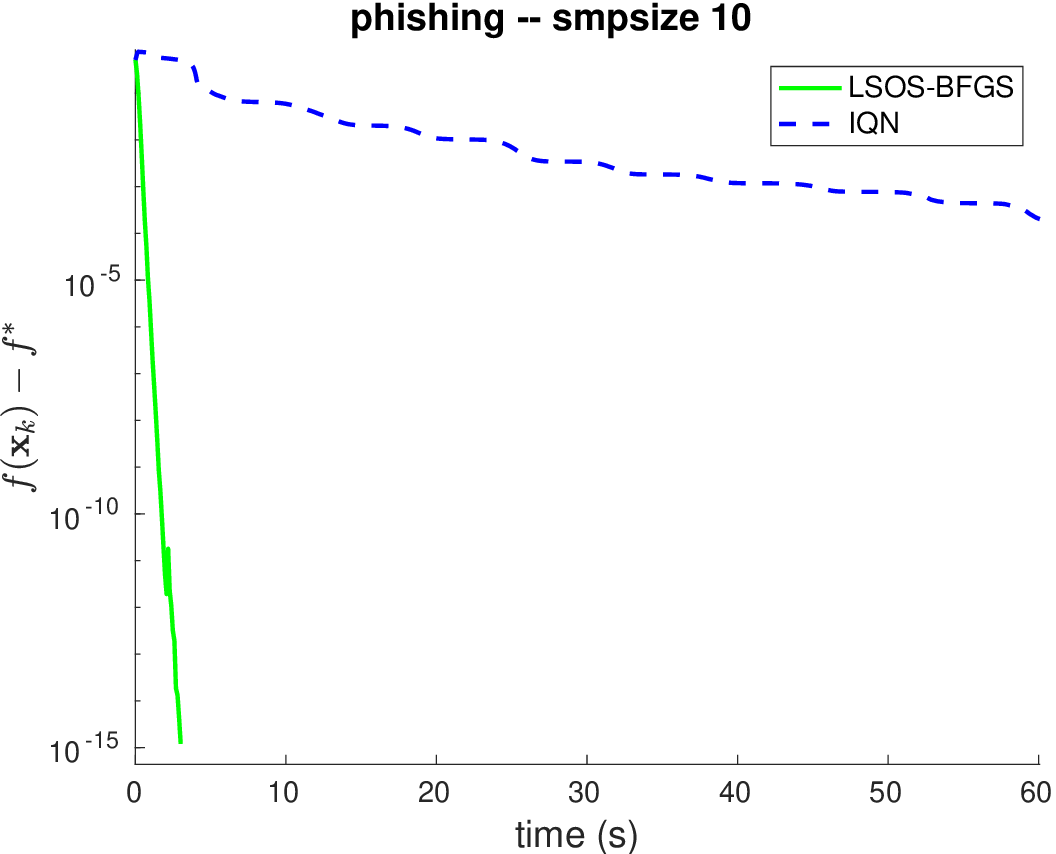}
    
    \caption{Binary classification problems: comparison of LSOS-BFGS and IQN in terms of function error versus execution time.\label{fig:binary_class_iqn_times}}
\end{figure}
\begin{figure}[h!]
    \includegraphics[width=0.47\textwidth]{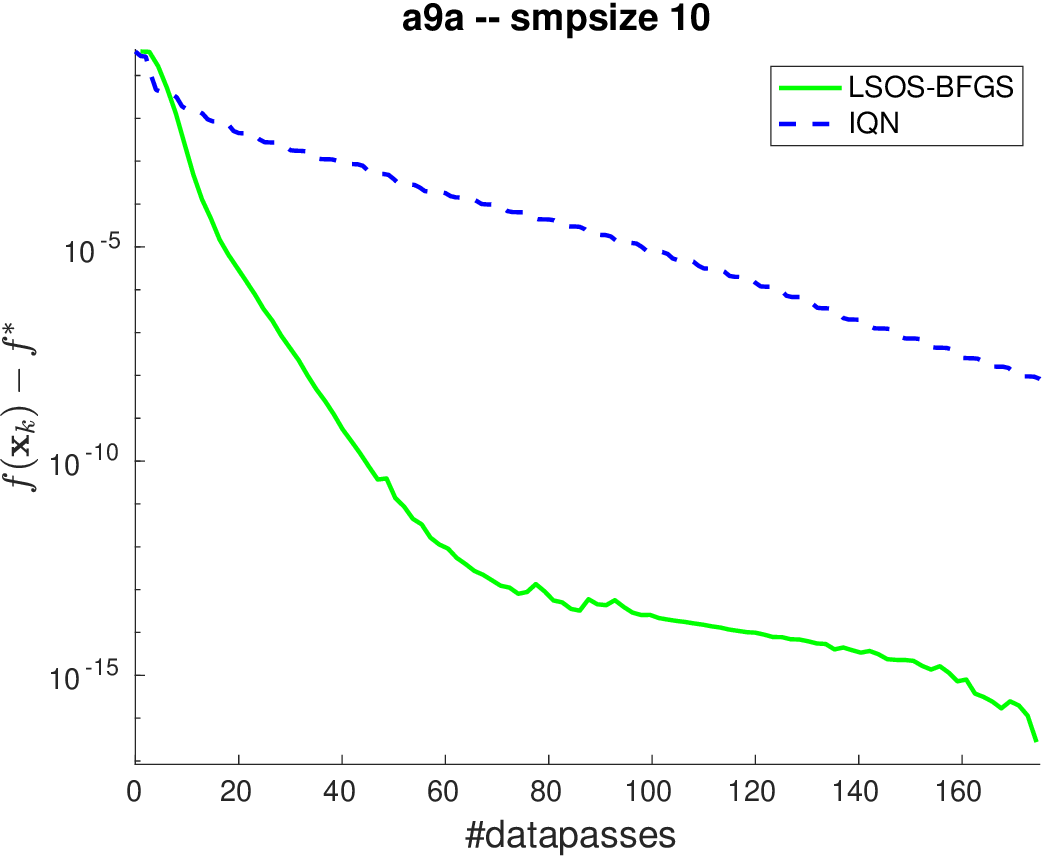}\includegraphics[width=0.47\textwidth]{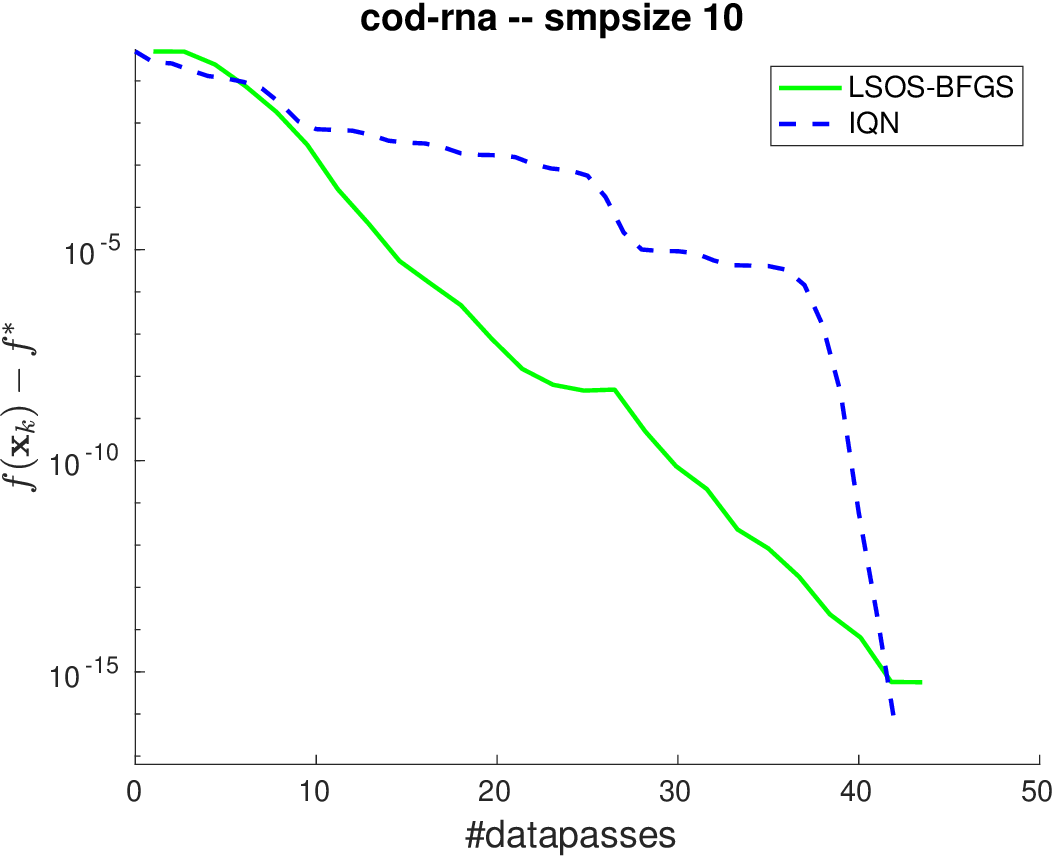}\\
    \includegraphics[width=0.47\textwidth]{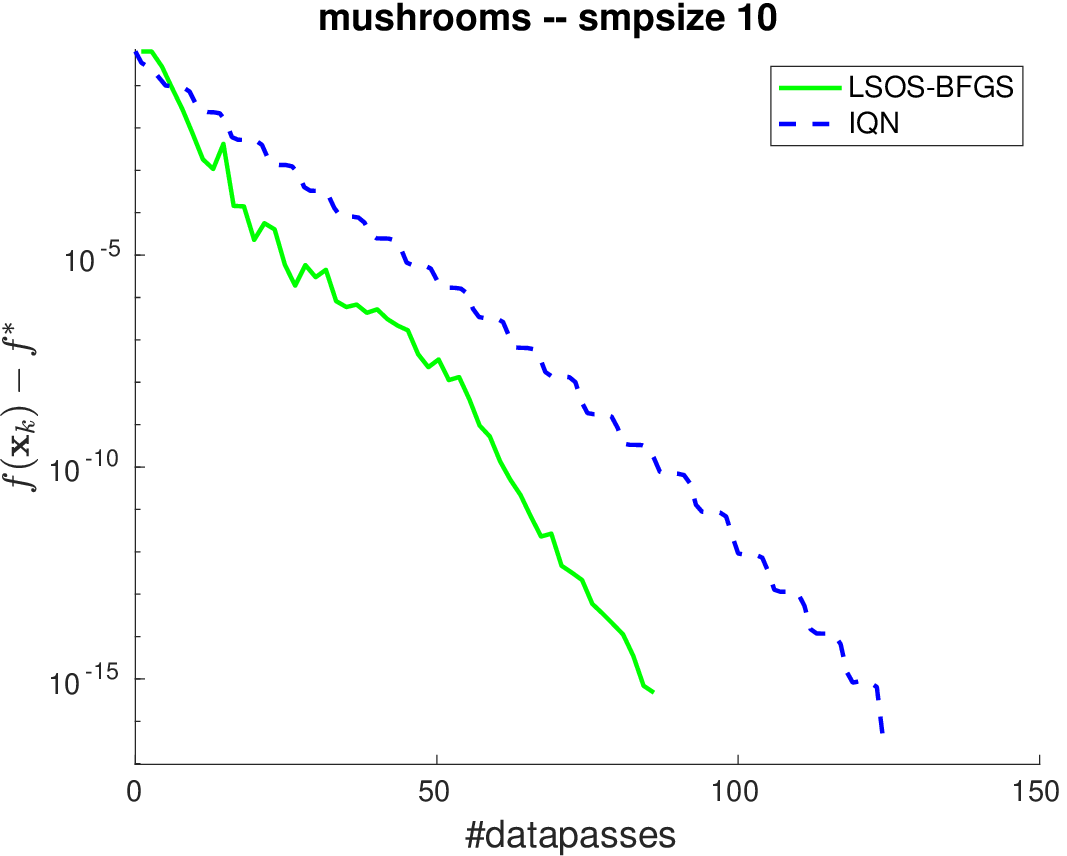}
    \includegraphics[width=0.47\textwidth]{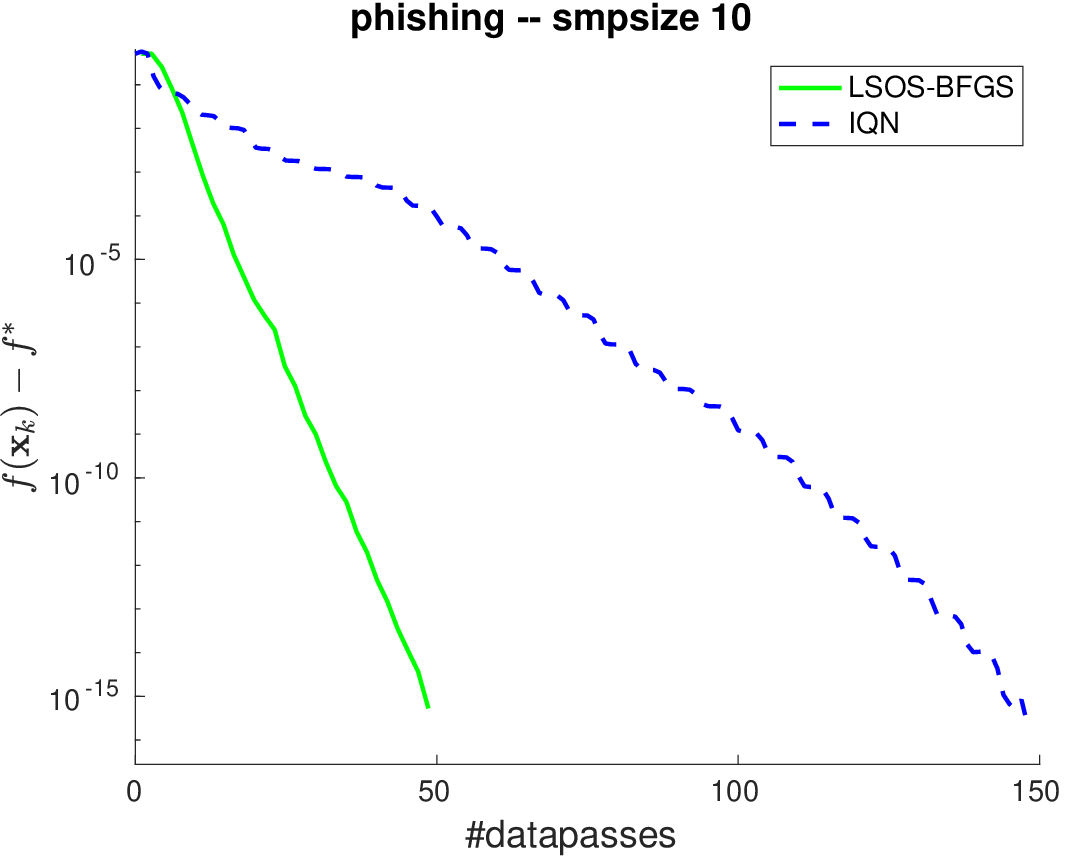}
    \caption{Binary classification problems: comparison of LSOS-BFGS and IQN in terms of function error versus data passes.\label{fig:binary_class_iqn_epochs}}
\end{figure}

Figures~\ref{fig:binary_class_iqn_times} and \ref{fig:binary_class_iqn_epochs} show a comparison between LSOS-BFGS and IQN in terms of the average absolute error on the objective function (with respect to the optimal value computed with the L-BFGS code by Mark Schmidt) versus the average execution time and the number of data passes, respectively. {It is worth mentioning that for IQN a data pass corresponds exactly to an epoch.} For the tests reported in Figure~\ref{fig:binary_class_iqn_times} we set the maximum execution time for the two algorithms equal to 60 seconds. For the tests reported in Figure~\ref{fig:binary_class_iqn_epochs}, we first run LSOS-BFGS for 60 seconds and then run IQN {to perform around as many data passes as the ones} completed by LSOS-BFGS in that time frame. Like the other cases, the values plotted for LSOS-BFGS are averaged over 20 runs. The shaded areas corresponding to the 95\% confidence interval are almost invisible, indicating low variance in the results. Since IQN spans the sample set cyclically, picking one sample at each iteration, IQN was run once for each test.

From Figure~\ref{fig:binary_class_iqn_times} LSOS-BFGS appears more efficient in terms of execution time for all the problems. This is possibly due to the cost which has to be paid to obtain the theoretical superlinear convergence rate in IQN, which requires at each iteration the solution of a dense linear system. Interestingly, despite the superlinear convergence rate of IQN, LSOS-BFGS is also able to outperform it when the comparison is made {in terms of oracle complexity} (Figure~\ref{fig:binary_class_iqn_epochs}). This suggests that, although we cannot prove a superlinear rate of convergence for the algorithms fitting into the LSOS framework, this approach can yield efficient algorithms in practice.

\section{Conclusions and outlook\label{sec:conclusions}}

We introduced a novel stochastic line-search algorithmic framework called LSOS, for the solution of nonconvex finite-sum problems, which allows the use of inexact second-order directions. Almost sure convergence to a stationary point for all the algorithms fitting into the LSOS framework was proved. Moreover, for strongly convex problems, we proved a.s.~convergence of the sequence of iterates to the unique minimizer. Numerical experiments showed that combining stochastic L-BFGS Hessian approximations with the SAGA variance-reduction technique and with line searches produces methods that are highly competitive with state-of-the art first- and second-order stochastic optimization methods {both when accounting for computational time and when accounting for oracle complexity}.

Future work will be focused on the application of methods from the LSOS class to nonconvex problems arising in the training of neural networks. Moreover, we intend to investigate the extension to the stochastic setting of the strategies for combining first- and second-order directions proposed in \cite{diserafino:2020lncs,diserafino:2020amc}. Finally, a challenging future research agenda includes the extension of (some of) the results presented in this work to constrained problems.

\medskip
\noindent{\bf Acknowledgement.} We are thankful to the associate editor and two anonymous referees whose comments helped us improve the paper. 

\providecommand{\bysame}{\leavevmode\hbox to3em{\hrulefill}\thinspace}
\providecommand{\MR}{\relax\ifhmode\unskip\space\fi MR }
% \MRhref is called by the amsart/book/proc definition of \MR.
\providecommand{\MRhref}[2]{%
    \href{http://www.ams.org/mathscinet-getitem?mr=#1}{#2}
}
\providecommand{\href}[2]{#2}

\end{document}